\documentclass[12pt]{amsart}
\usepackage{amssymb, tensor, fullpage, textcomp, pb-diagram, wasysym, wrapfig, comment, xspace, leftindex, xcolor, appendix, tensor, relsize}
\usepackage[Symbol]{upgreek}
\usepackage[mathscr]{euscript}
\usepackage[shortlabels]{enumitem}

\let\oldtocsection=\tocsection

\let\oldtocsubsection=\tocsubsection

\let\oldtocsubsubsection=\tocsubsubsection

\renewcommand{\tocsection}[2]{\hspace{0em}\oldtocsection{#1}{#2}}
\renewcommand{\tocsubsection}[2]{\hspace{1em}\oldtocsubsection{#1}{#2}}
\renewcommand{\tocsubsubsection}[2]{\hspace{2em}\oldtocsubsubsection{#1}{#2}}

\DeclareFontFamily{U}{fsy}{}
\DeclareFontShape{U}{fsy}{m}{n}{<->s*[.9]psyr}{}
\DeclareSymbolFont{der@m}{U}{fsy}{m}{n}
\DeclareMathSymbol{\der}{\mathord}{der@m}{182}

\newcommand{\acl}{\operatorname{acl}}

\newcommand{\nset}{\{1, \ldots, n \}}

\newcommand{\I}{\mathbb{I}_u}

\newcommand{\cyc}{C}

\newcommand{\V}{\mathbb{V}}
\newcommand{\D}{\mathbb{D}}

\newcommand{\pac}{\mathrm{PAC}}

\newcommand{\dlo}{\mathrm{DLO}}

\newcommand{\dprk}{\operatorname{dp}}

\newcommand{\lex}{<_{\mathrm{Lex}}}

\DeclareSymbolFont{imag@m}{OT1}{cmr}{m}{ui}
\DeclareMathSymbol{\imag}{\mathord}{imag@m}{105}

\newtheorem{theorem}{Theorem}[section]
\newtheorem*{theorem*}{Theorem}

\newtheorem*{qst*}{Question}

\newtheorem{proposition}[theorem]{Proposition}
\newtheorem*{proposition*}{Proposition}
\newtheorem*{fact*}{Fact}
\newtheorem{lem}[theorem]{Lemma}

\newtheorem*{Claim*}{Claim}
\newtheorem{fact}[theorem]{Fact}

\newtheorem{lemma}[theorem]{Lemma}
\newtheorem{corollary}[theorem]{Corollary}

\newtheorem*{thmA}{Theorem A}
\newtheorem*{factA}{Fact A}
\newtheorem*{thmB}{Theorem B}
\newtheorem*{thmC}{Theorem C}
\newtheorem*{thmD}{Theorem D}

\theoremstyle{definition}

\theoremstyle{remark}

\newcommand{\monster}{\boldsymbol{\Sa M}}

\newcommand{\monsterset}{\boldsymbol{M}}

\newcommand{\mfrak}{\mathfrak{m}}
\newcommand{\Val}{\mathrm{Val}}

\newcommand{\hgroup}{(H;+,\triangleleft)}
\newcommand{\rgoup}{(\R;+,<)}
\newcommand{\zgoup}{(\Z;+,<)}
\newcommand{\jgroup}{(J;+,C)}
\newcommand{\rfield}{(\R;+,\times)}
\newcommand{\rcf}{\mathrm{RCF}}

\newcommand{\rad}{\operatorname{rad}}

\newcommand{\rvec}{\R_{\mathrm{Vec}}}

\newcommand{\Th}{\mathrm{Th}}
\ProvideTextCommandDefault{\cprime}{(U+042C)}

\newcommand{\Fraisse}{Fra\"iss\'e\xspace}
\newcommand{\Erdos}{Erd\H{o}s\xspace}

\newcommand{\st}{\operatorname{st}}

\newenvironment{claimproof}[1][\proofname]
               {
                 \proof[#1]
                 
               }
               {
                 \endproof
               }

\newcommand{\cl}{\operatorname{Cl}}

\newcommand{\shk}{\K^\mathrm{Sh}}

\newcommand{\Sh}[1]{\ensuremath{\mathscr{#1}^{\mathrm{Sh}}}}
\newcommand{\Sq}[1]{\ensuremath{\mathscr{#1}^{\square}}}
\newcommand{\nip}{\mathrm{NIP}}
\newcommand{\ip}{\mathrm{IP}}

\newcommand{\Cal}[1]{\ensuremath{\mathcal{#1}}}
\newcommand{\Sa}[1]{\ensuremath{\mathscr{#1}}}

\newcommand{\B}{\mathbb{B}}
\newcommand{\Z}{\mathbb{Z}}
\newcommand{\N}{\mathbb{N}}

\newcommand{\Q}{\mathbb{Q}}
\newcommand{\R}{\mathbb{R}}
\newcommand{\F}{\mathbb{F}}
\newcommand{\E}{\mathbb{E}}
\newcommand{\K}{\mathbb{K}}

\begin{document}
\title[]{Trace definability II: model-theoretic linearity}

\author{Erik Walsberg}
\email{erik.walsberg@gmail.com}

\begin{abstract}
We give examples of $\nip$ structures in which new algebraic structure appears in the Shelah completion.
In particular we construct a weakly o-minimal structure $\Sa M$ such that $\Sa M$ does not interpret an infinite group but the Shelah completion of $\Sa M$ interprets an infinite field.
We introduce a weak notion of interpretability called local trace definability between first order structures and an associated weak notion of equivalence.
We give a dichotomy between ``linearity" and ``field structure" for dp-minimal expansions of archimedean ordered abelian groups.
We also prove several other results about trace definability and local trace definability between various classes of structures.
\end{abstract}

\maketitle

\section*{Introduction}
A theme in model theory is to show that some abstract notion of ``linearity" or ``triviality" is equivalent to non-interpretability of some algebraic structure.
Two examples:
\begin{enumerate}[leftmargin=*]
\item Zil'ber showed that a totally categorical structure does not interpret an infinite group if and only if every interpretable strongly minimal set is disintegrated~\cite{Zilber_1993}.
\item Peterzil and Starchenko showed that an o-minimal structure does not interpret an infinite group (field) if and only it is disintegrated (locally modular)~\cite{PS-Tri}.
\end{enumerate}

Fact~A is a special case of the Peterzil-Starchenko theorem.

\begin{factA}
If $\Sa M$ is an o-minimal expansion of an ordered abelian group $(M; +, <)$ then the following are equivalent.
\begin{enumerate}[leftmargin=*]
\item $\Sa M$ is locally modular.
\item $\Sa M$ does not interpret an infinite field.
\item There is an ordered division ring $\D$ and an ordered $\D$-vector space $\Sa V$ expanding $(M; +, <)$ such that every $\Sa M$-definable set is definable in $\Sa V$, i.e. every $\Sa M$-definable set is a boolean combination of solution sets of $\D$-linear inequalities.
\end{enumerate}
\end{factA}

We are interested in obtaining similar results over other classes of unstable $\nip$ structures.
Let $\Sa M$ range over $\nip$ structures.
We let $\Sh M$ be the {\bf Shelah completion} of $\Sa M$, i.e. the expansion of $\Sa M$ by all externally definable sets.
(This is usually called the Shelah expansion.)
See Section~\ref{section:new} for background.
In general $\nip$-theoretic properties transfer immediately from $\Sa M$ to $\Sh M$.
Thus if non-interpretability of an infinite group (field) is equivalent to a $\nip$-theoretic property over some reasonable class $\Cal C$ of $\nip$ structures then we expect that the Shelah completion of any $\Sa M \in \Cal C$ interprets an infinite group (field) if and only if $\Sa M$ does.
So we were surprised to find that new algebraic structure can appear in the Shelah completion.

\begin{thmA}
\hspace{.0000000000000000000000000000000000000000000000000000000000000000000000000000000000000000000000000000000000000000001cm}
\begin{enumerate}[leftmargin=*]
\item There is a weakly o-minimal structure $\Sa M$ such that $\Sa M$ is disintegrated, no model of $\Th(\Sa M)$ interprets an infinite group, but the Shelah completion of $\Sa M$ interprets the field of real numbers.
\item There is a weakly o-minimal expansion $\Sa Q$ of $(\Q;+,<)$ such that no model of $\Th(\Sa Q)$ interprets an infinite field, but the Shelah completion of any $\aleph_1$-saturated elementary extension of $\Sa Q$ interprets the field of real numbers.
\item Given a $p$-adically closed field $\K$ let $\Sa B_\K$ be the induced structure on the set of balls in $\K$, so $\Sa B_\K$ is interpretable in $\K$.
Then $\Sa B_\K$ does not interpret an infinite field but the Shelah completion of $\Sa B_\K$ interprets the field of $p$-adic numbers when $\K$ is $\aleph_1$-saturated.
\end{enumerate}
\end{thmA}

See Section~\ref{section:new} for the details of these examples.
How can the situation be salvaged?
We believe that the right approach is to work with a weaker notion of interpretability.
In \cite{trace1} we introduced a weak notion of interpretability called trace definability, see Section~\ref{section:bckgrnd} for background.
This gives rise to a weak notion of equivalence which we call trace equivalence.
Any $\nip$ structure is trace equivalent to its Shelah completion.
It follows that $\Th(\Sa M)$ trace defines a structure $\Sa O$ when the Shelah completion of $\Sa M$ interprets $\Sa O$.
At present we also introduce an even weaker notion of interpretability, local trace definability, which again gives rise to a weaker notion of equivalence called local trace equivalence.
We use these notions to give a dichotomy between ``linearity" and ``field structure" for dp-minimal expansions of archimedean ordered abelian groups.
(An expansion of $\rgoup$ is dp-minimal if and only if it is o-minimal and an expansion of a divisible archimedean ordered abelian group is dp-minimal if and only if it is weakly o-minimal~\cite{Simon-dp, SW-dp}.)

\begin{thmB}
Let $\Sa R$ be a dp-minimal expansion of an archimedean ordered abelian group $(R;+,<)$, which we take to be a substructure of $\rgoup$.
Then the following are equivalent.
\begin{enumerate}[leftmargin=*]
\item The Shelah completion of a model of $\Th(\Sa R)$ cannot interpret $\rfield$.
\item $\Th(\Sa R)$ does not locally trace define an infinite field.
\item $\Sa R$ has near linear Zarankiewicz bounds.
\item Every $\Sa R$-definable set is a boolean combination of $(R;+)$-definable sets and sets of the form $Y \cap R^n$ where $Y \subseteq \R^n$ is the solution set of a system of linear inequalities.
\end{enumerate}
\end{thmB}

Near linear Zarankiewicz bounds is a candidate for a $\nip$-theoretic notion of ``linearity".
We introduce this property, show that it implies $\nip$, is preserved under local trace definability, and rules out local trace definability of characteristic zero rings and infinite rings without zero divisors.

\medskip
The proof of Theorem~B goes through the following result.

\begin{thmC}
Let $\Sa R$ be as in Theorem~B.
Then $\Sa R$ is trace equivalent to an o-minimal expansion $\Sa S$ of $\rgoup$ such that $\Sa S$ is also interpretable in the Shelah completion of any $\aleph_1$-saturated elementary extension of $\Sa R$.
Furthermore $\Sa R$ is either locally trace equivalent to $\rgoup$ or locally trace equivalent to some o-minimal expansion of $\rfield$.
\end{thmC}

The last part of Theorem~C reduces to showing that any o-minimal expansion of $\rgoup$ is either locally trace equivalent to $\rgoup$ or locally trace equivalent to an o-minimal expansion of $\rfield$.
We also show that any o-minimal expansion of an ordered abelian group is either locally trace equivalent to $\rgoup$ or locally trace equivalent to an o-minimal expansion of an ordered field.

\medskip
It follows from the proof of Theorem~C that any o-minimal expansion of an ordered group which does not interpret an infinite field is locally trace equivalent to $\rgoup$.
We show in Section~\ref{section:dis} that an o-minimal structure which does not interpret an infinite group is locally trace equivalent to $(\R;<)$.
This follows from the Peterzil-Starchenko trichotomy theorem and the second claim of Theorem~D below, which summarizes our other results on (local) trace definability.
We let $\Sa M \sqcup \Sa N$ be the disjoint union of structures $\Sa M$ and $\Sa N$, considered as a two-sorted structure in the natural way.
Recall that a structure is said to be weakly minimal if it is superstable of U-rank one.

\begin{thmD}
\hspace{.000000001cm}
\begin{enumerate}[leftmargin=*]
\item Any disintegrated weakly minimal structure is locally trace equivalent to an infinite set equipped with equality.
In particular any colored bounded degree graph is locally trace equivalent to an infinite set equipped with equality.
\item Any weakly quasi-o-minimal structure in a binary relational language is locally trace equivalent to $(\R;<)$.
In particular any colored linear order or disintegrated o-minimal structure is locally trace equivalent to $(\R; <)$.
\item Any non-divisible non-singular torsion free abelian group is trace equivalent to $(\Z; +)$.
Furthermore $\Th(\Z; +)$ is, with respect to trace definability, the minimal theory above $\Th(\Q; +)$ which is not totally transcendental.
\item The following structures are trace equivalent to $\rgoup$\textup{:} Any non-singular ordered abelian group, any infinite non-singular cyclically ordered abelian group, $(\R; +, <, \Z, \Q)$, and $(\R; +, <, t \mapsto \lambda t)$ for any irrational algebraic $\lambda \in \R$.
Any ordered vector space over an ordered division ring is locally trace equivalent to $\rgoup$.
\item Let $\hgroup$ be an ordered abelian group.
If any elementary extension of $\hgroup$ has only countably many definable convex subgroups then $\hgroup$ is locally trace equivalent to $(H; +) \sqcup \rgoup$.
If $\hgroup$ has only finitely many definable convex subgroups, in particular if $\hgroup$ is archimedean or has finite rank, then $\hgroup$ is trace equivalent to $(H; +) \sqcup \rgoup$.
\end{enumerate}
\end{thmD}

It follows from (5) that any ordered abelian group with bounded regular rank has near linear Zarankiewicz bounds and hence cannot locally trace define an infinite field.
The same holds for $(\R; +, <, \Z, \Q)$ by (4).
Finally, (3) is sharp in that a torsion free abelian group is non-singular if and only if it is superstable and superstability is preserved under trace definability.


\subsection*{Acknowledgments}
All of this began when John Goodrick asked me if there is a version of Fact~A for dp-minimal expansions of ordered abelian groups.
Thanks to Artem Chernikov for very useful conversations and for encouraging me to write up the example presented in Section~\ref{section;noal}.
This research was funded in part by the Austrian Science Fund (FWF) 10.55776/PAT1673125.

\section*{Conventions}
Throughou $m,n,k$ are natural numbers.
All structures, languages, and theories are first order and all theories are complete and consistent unless noted otherwise.
We generally assume that all structures are infinite and all theories are theories of infinite structures.
``Definable" without modification means ``first order definable, possibly with  parameters" and ``zero-definable" means ``definable without parameters".
If $\Sa M$ and $\Sa M^*$ are structures on the same domain then $\Sa M$ is a {\bf reduct} of $\Sa M^*$ if every $\Sa M$-definable set is $\Sa M^*$-definable and $\Sa M, \Sa M^*$ are {\bf interdefinable} if each is a reduct of the other.
Two structures on possibly different domains are {\bf bidefinable} if they are interdefinable up to isomorphism.
Throughout $\Sa M$ is a structure with domain $M$.
The {\bf structure induced} on $A \subseteq M^n$ by $\Sa M$ is the structure with an $m$-ary relation defining $Y \cap A^n$ for every $\Sa M$-definable $Y \subseteq M^{mn}$.
Note that this induced structure admits quantifier elimination if and only if every definable set is of the form $Y \cap A^m$ for $\Sa M$-definable $Y \subseteq M^{mn}$.

\medskip
A theory is {\bf disintegrated} if the algebraic closure of any subset of any model is the union of the algebraic closures of the elements of that subset.
A structure is disintegrated when its theory is.

\medskip
We let $|y|$ be the length of a tuple $y$ of variables.
Given a formula $\varphi(x,y)$ in a structure $\Sa M$ we let $\varphi(M^{|x|}, \beta)$ be the subset of $M^{|x|}$ defined by $\varphi(x, \beta)$ for any $\beta \in M^{|y|}$.
Given sets $A,B$ and a subset $X$ of $A \times B$ we let $X_a$ be the set of $b \in B$ such that $(a,b) \in X$ for any $a \in A$.
An $\Sa M$-definable family of sets is a collection of sets of the form $\{X_a : a \in M^m\}$ for some $\Sa M$-definable $X \subseteq M^m \times M^n$.

\medskip
We let $\cl(X)$ be the closure of a subset $X$ of a topological space.

\medskip
A {\bf finitely homogeneous} structure is a countable structure in a finite relational language with quantifier elimination.
Equivalently, a finitely homogeneous structure is a \Fraisse limit of a \Fraisse class of structures in a finite relational language.

\section{New algebraic structure in the Shelah expansion}\label{section:new}
We recall some background on externally definable sets and the Shelah completion.
Throughout this section $\Sa M$ is a $\nip$ structure.
A subset of $M^m$ is {\bf externally definable} if it is of the form $Y \cap M^m$ for a definable subset $Y \subseteq N^m$ with $\Sa N$ an elementary extension of $\Sa M$.

\begin{fact}
\label{fact:convex}
Suppose that $X$ is an $\Sa M$-definable set and $<$ is an $\Sa M$-definable linear order on $X$.
Then any $<$-convex subset of $X$ is externally definable.
\end{fact}

Fact~\ref{fact:convex} follows by the definitions and an obvious compactness argument.
It also follows by
Fact~\ref{fact:cs}.
The latter is due to Chernikov and Simon~\cite[Prop.~3.21]{Simon-Book}.
It requires $\nip$.

\begin{fact}\label{fact:cs}
A subset $X\subseteq M^n$ is externally definable if and only if there is a definable family $\Cal X$ of subsets of $M^n$ so that for every finite $A \subseteq X$ we have $A \subseteq X' \subseteq X$ for some $X' \in \Cal X$.
\end{fact}

We now introduce some non-standard terminology.
A structure is {\bf Shelah complete} if every externally definable set is already definable.
The well-known characterization of stability in terms of definable types shows that a theory is stable if and only if all of its models are Shelah complete.
The Marker-Steinhorn theorem shows that any o-minimal expansion of $(\R;<)$ is Shelah complete~\cite{Marker-Steinhorn} and Delon showed that the field $\Q_p$ is Shelah complete~\cite{Delon-def}.
The \textbf{Shelah completion} $\Sh M$ of $\Sa M$ is the expansion of $\Sa M$ by all externally definable sets.
Up to interdefinability $\Sh M$ is the structure induced on $M$ by any $|M|^+$-saturated elementary extension of $\Sa M$.
(This is usually referred to as the \textit{Shelah expansion}, we prefer \textit{completion} as it is more descriptive and \textit{expansion} is already extensively used in model theory.)
Facts~\ref{fact:shelah} and \ref{fact:cs} together imply that the Shelah completion of a $\nip$ structure is Shelah complete.

\begin{fact}
\label{fact:shelah}
Every $\Sh M$-definable set is externally definable.
\end{fact}

Fact~\ref{fact:shelah} is due to Shelah~\cite{Shelah-external}.
It is also an easy consequence of Fact~\ref{fact:cs}.

\begin{fact}\label{fact:ksh}
If $\Sa N$ is an elementary extension of $\Sa M$ then the structure induced on $M$ by $\Sh N$ is interdefinable with $\Sh M$.
\end{fact}

Fact~\ref{fact:ksh} is an easy consequence of Fact~\ref{fact:shelah}.
We now present our examples.

\subsection{The induced structure on rational powers of $\lambda$}
Let $H$ be a non-trivial finite rank divisible subgroup of $(\R;+)$ and fix a real number $\lambda > 1$.
For example, take $H = \Q$ and $\lambda = 2$.
Let $\Sa H_\lambda$ be the structure on $H$ with an $n$-ary relation defining \[\{(h_1,\ldots,h_n)\in H^n: (\lambda^{h_1},\ldots,\lambda^{h_n})\in X \}\] for every semialgebraic $X\subseteq\R^n$.
Hence $\Sa H_\lambda$ is bidefinable with the structure induced on $\{\lambda^h : h\in H\}$ by $\rfield$.
Note that $(H;+,<)$ is a reduct of $\Sa H_\lambda$.
Fact~\ref{fact:vddgun} is due to van den Dries and G{\"u}naydin~\cite{DrGu}.

\begin{fact}\label{fact:vddgun}
A subset of $H^m$ is definable in $\Sa H_\lambda$ if and only if it is the preimage of a semialgebraic subset of $\R^m$ under the map $H^m \to \R^m$ given by $(h_1,\ldots,h_m) \mapsto (\lambda^{h_1},\ldots,\lambda^{h_m})$, i.e. $\Sa H_\lambda$ admits quantifier elimination.
\end{fact}

Fact~\ref{fact:vddgun} shows in particular that $\Th(\Sa H_\lambda)$ is weakly o-minimal.

\begin{proposition}
\label{prop:mann}
An elementary extension of $\Sa H_\lambda$ cannot interpret an infinite field.
\end{proposition}

\begin{proof}
By a theorem of Eleftheriou~\cite{E-small} $\Th(\Sa H_\lambda)$ eliminates imaginaries.
Berenstein and Vassiliev~\cite[Prop.~3.16]{BV-one-based} show that $\Th(\Sa H_\lambda)$ is weakly one-based and a model of a weakly one-based theory cannot define an infinite field by~\cite[Prop.~2.11]{BV-one-based}.
\end{proof}

\begin{proposition}\label{prop:hlamb}
Let $\Sa M$ be an $\aleph_1$-saturated elementary extension of $\Sa H_\lambda$.
Then $\Sh M$ interprets $\rfield$.
\end{proposition}

\begin{proof}
We replace $\Sa H_\lambda$ with the structure $\Sa J$ induced on $J = \{\lambda^h : h\in H\}$ by $\rfield$.
Let $W$ be the convex hull of $J$ in $M$ and $\mfrak$ be the set of $\alpha \in M$ such that $\lambda^{-1/n} < \alpha < \lambda^{1/n}$ for every $n\ge 1$.
Then $W$ and $\mfrak$ are convex subgroups of $M$ and are hence $\Sh M$-definable.
We identify $W/\mfrak$ with $\R_>$ and the quotient map $W \to \R_>$ with the standard part map.
We show that addition and multiplication on $\R_>$ are $\Sh M$-definable.
Let $X \subseteq \R^3$ be either $\{(a,b,c) \in \R^3_> : a + b \le c\}$ or $\{(a,b,c) \in \R^3_> : ab \le c\}$.
It suffices to show that $X$ is definable.
Let $Y = X \cap J^3$ and note that $Y$ is $\Sa J$-definable and dense in $X$.
It easily follows that $X = \st(Y^* \cap W^3)$ for $Y^*$ the canonical extension of $Y$ to an $\Sa M$-definable set.
\end{proof}

\subsection{The induced structure on an algebraically independent subset of $\R$}\label{section;noal}
In this section $\Sa R$ is an o-minimal expansion of $\rgoup$, $\Sa N$ is an $\aleph_1$-saturated elementary extension of $\Sa R$, $P$ is a dense $\acl$-independent subset of $N$, and $\Sa P$ is the structure induced on $P$ by $\Sa N$.
For example we could take $\Sa N$ to be an $\aleph_1$-saturated real closed field and $P$ to be dense algebraically independent subset of $N$.

\begin{proposition}
\label{prop:newdm}
The theory of $\Sa P$ is weakly o-minimal and disintegrated.
Furthermore $\Sh P$ interprets $\Sa R$ but a model of $\Th(\Sa P)$ cannot interpret an infinite group.
\end{proposition}

Fact~\ref{fact:dms} is due to Dolich, Miller, and Steinhorn~\cite[2.16]{DMS-Indepedent}.

\begin{fact}\label{fact:dms}
The structure $\Sa P$ admits quantifier elimination.
Equivalently: a subset of $P^m$ is definable in $\Sa P$ if and only if it is of the form $Y \cap P^m$ for $\Sa N$-definable $Y \subseteq P^m$.
\end{fact}

It follows that $\Th(\Sa P)$ is weakly o-minimal and that we have $\acl(A) = P \cup A$ for any subset $A$ of an elementary extension of $\Sa P$.
In particular $\Th(\Sa P)$ is disintegrated.
A theorem of Eleftheriou~\cite[Thm.~C]{Elef-small-sets} shows that $\Th(\Sa P)$ eliminates imaginaries and a result of Berenstein and Vassiliev~\cite[Cor.~6.3]{BV-independent} shows that a model of $\Th(\Sa P)$ cannot define an infinite group.
We need Lemma~\ref{lem:regular} below to show that $\Sh P$ interprets $\Sa R$.
Recall that a subset of a topological space is {\bf regular open} if it is the interior of its closure.

\begin{lemma}
\label{lem:regular}
Every $\Sa R$-definable set is a boolean combination of regular open $\Sa R$-definable sets.
\end{lemma}

It is easier to show that $\rfield$ is interpretable in $\Sh P$ when $\Sa R$ expands $\rfield$, as 
$$ \{ (a, b, c) \in \R^3 : a + b < c\} \quad \text{and} \quad \{ (a, b, c) \in \R^3 : ab < c \}  $$
are both regular open.
We say that $\Sa R$ defines a global field structure if there are definable functions $\oplus,\otimes: \R^2 \to \R$ such that $(\R,<,\oplus,\otimes)$ is isomorphic to $(\R,<,+,\cdot)$.

\begin{proof}
An application of o-minimal cell decomposition shows that every definable set is a boolean combination of open definable sets.
It is therefore enough to show that any definable open set is a finite union of definable regular open sets.
It is easy to see that open cells are regular. 
Edmundo, Eleftheriou, and Prelli~\cite{EEL-covering} show that if $\Sa R$ does not define a global field structure then any definable open set is a finite union of open cells.

\medskip
Suppose that $\Sa R$ defines a global field structure and let $U$ be a definable open subset of $\R^n$.
Without loss of generality we suppose $\Sa R$ expands $\rfield$.
Let $B_t$ be the open ball in $\R^n$ with center the origin and radius $t > 0$.
It suffices to show that $U \cap B_2$ and $U \setminus \cl(B_1)$ are both finite unions of definable regular open sets.
Let $\iota \colon \R^n \setminus \{0\} \to \R^n \setminus \{0\}$ be given by $\iota(v) = v/\|v\|$, i.e. $\iota$ is the inversion across the unit sphere.
Now $\iota$ is a definable homeomorphism and hence takes definable regular open sets to definable regular open sets.
Wilkie~\cite{Wilkie-covering} shows that any definable bounded open set is a finite union of open cells.
So $U \cap B_2$ is a finite union of open cells.
Furthermore $\iota(U \setminus \cl(B_1)) \subseteq B_1$ is a union of open cells $V_1,\ldots,V_m$.
So $U \setminus \cl(B_1)$ is the union of $\iota(V_1),\ldots,\iota(V_m)$.
\end{proof}

We now show that $\Sh P$ interprets $\Sa R$.
Let $O$ be the set of $a \in P$ such that $|a| < n$ for some $n$.
Let $Q$ be the set of $a \in P$ such that $\frac{1}{n} < a$ for some $n \ge 1$.
Then $O$ and $Q$ are both convex, hence definable in $\Sh P$.
 Let $E$ be the equivalence relation on $P$ where $(a,b) \in E$ when $|a - b| < 1/n$ for all $n > 0$.
We show that $E$ is definable in $\Sh P$.
Let $X$ be the set of $(a,b,c) \in N^3$ such that $| a - b | < c$.
Then $C := X \cap P^3$ is definable in $\Sa P$.
Observe that
$$ E = \bigcap_{c \in Q} \{ (a,b) \in P^2 : (a,b,c) \in C \} $$
so $E$ is definable in $\Sh P$.
By density and saturation we can identify $O/E$ with $\R$ and identify the standard part map $\st \colon O \to \R$ with the quotient map $O \to O/E$.
As $\Sh P$ defines the usual order on $\R$, it defines a basis for the topology on $\R^n$.
We show that $\Sa R$ is a reduct of the structure induced on $\R$ by $\Sh P$.
By Lemma~\ref{lem:regular} it suffices to suppose that $U \subseteq \R^n$ is regular, open, and $\Sa R$-definable and show that $U$ is definable in $\Sh P$.
Let $U^*$ be the subset of $N^n$ defined by any formula defining $U$.
Now $U^* \cap P^n$ is $\Sh P$-definable and dense in $U^*$.
Furthermore $U^* \cap O^n = (U^* \cap P^n) \cap O^n$ is $\Sh P$-definable.
It is easy to see that $\cl(U) = \st(U^* \cap O^n)$, so $\cl(U)$ is definable in $\Sh P$.
Finally $U$ is $\Sh P$-definable as $U$ is the interior of $\cl(U)$ by regularity.



\subsection{The induced structure on $p$-adic balls}\label{section:p-adic}
Fix a prime $p$.
We refer to \cite{Belair-panorama} for basic facts on the model theory of $p$-adically closed fields.
Let $\B$ be the set of balls in $\Q_p$, considered as a $\Q_p$-definable set of imaginaries.
Let $\Sa B$ be the structure induced on $\B$ by $\Q_p$.
Given an elementary extension $\K$ of $\Q_p$ let $\Sa B_\K$ be the associated elementary extension of $\Sa B$.
We first show that $\Sa B_\K$ does not interpret an infinite field.

\begin{fact}
\label{fact:obv}
Any infinite field interpretable in a $p$-adically closed field $\K$ is definably isomorphic to a finite extension of $\K$.
\end{fact}

Fact~\ref{fact:obv}  is due to Halevi, Hasson, and Peterzil~\cite{halevi-hasson-peterzil}.
Suppose that $\Sa B_\K$ interprets an infinite field $F$.
By Fact~\ref{fact:obv} we may suppose that the domain of $F$ is some $\K^n$, so there is a $\K$-definable surjection $\B_\K^m \to \K^n$ for some $m,n \ge 1$.
This contradicts Lemma~\ref{lem:map}.

\begin{lem}
\label{lem:map}
Let $\K$ and $\B_\K$ be as above.
Any definable function $\B_\K^m \to \K^n$ has finite image.
\end{lem}

\begin{proof}
Recall that a definable subset of $\K$ is finite if and only if it has empty interior.
Hence $\K$ eliminates $\exists^\infty$, so it suffices to prove the lemma over $\Q_p$.
If $f \colon \B^m \to \Q_p^n$ has infinite image then there is a coordinate projection $e \colon \Q^m_p \to \Q_p$ so that $e \circ f$ has infinite image.
Hence it suffices to show that  the image of any definable function $\B^m \to \Q_p$ has empty interior.
This holds as $\B$ is countable and any nonempty open subset of $\Q_p$ is uncountable.
\end{proof}

\begin{proposition}\label{prop:p-adic}
Suppose that $\K$ is an $\aleph_1$-saturated elementary extension of $\Q_p$.
Then $\Sh B_\K$ interprets $\Q_p$.
\end{proposition}

\begin{proof}
Let $\Gamma$ be the value group of $\K$ with respect to the $p$-adic valuation, this is an elementary extension of $\zgoup$.
Now $\Z$ is a convex subgroup of $\Gamma$ so we let $\Val$ be the natural valuation $\K^\times \to \Gamma/\Z$ and let $V$ be the valuation ring of $\Val$.
Then the residue field of $\Val$ is $\Q_p$ and the residue map $\st \colon V \to \Q_p$ is the standard part map that takes every ``finite" element of $\K$ to the nearest element of $\Q_p$.
We also let $\st\colon V^n\to \Q_p^n$ be the map given by applying $\st$ coordinate-wise for each $n\ge 2$. 

\begin{Claim*}
Fix $f\in \Z[x_1,\ldots,x_n]$ and let $X\subseteq \Q_p^n$, $X^*\subseteq \K^n$ be the vanishing set of $f$ over $\Q_p$, $\K$, respectively.
Then $\st(X^* \cap V^n)=X$.
\end{Claim*}

\begin{claimproof}
First note that $\st(X^*\cap V^n)$ is contained in $X$ as $\st$ is a ring homomorphism.
The inclusion $\Q_p\to \K$ is right-inverse to $\st$.
Hence $X \subseteq X^*\cap V^n$, and so $\st(X^*\cap V^n)$ contains $\st(X) = X$.
\end{claimproof}

An application of Fact~\ref{fact:ksh} shows that $\Sh B_\K$ is interdefinable with the structure induced on $\B$ by $\shk$.
Let $B(a, \delta) \in \B_\K$ be the ball with center $a$ and radius $\delta$.
We let $\rad(B) \in \Gamma$ be the radius of a ball $B \in \B_\K$.
Then $\rad \colon \B_\K \to \Gamma_>$ is surjective and $\K$-definable, so we consider $\Gamma_>$ to be an imaginary sort of $\Sa B_\K$ and $\rad$ to be a $\Sa B_\K$-definable function.
Fix $\delta \in \Gamma_>$ with $\delta > \N$.
Let $D$ be the set of balls $B(a,\delta) \in \B_\K$ with $a \in V$.
Now $\Val$ is $\shk$-definable as $\Z$ is a convex subset of $\Gamma$. 
Hence $V$ is $\shk$-definable.
It follows that $D$ is $\Sh B_\K$-definable.
Note that for any $\alpha \in V$ and $\beta \in B(\alpha,\delta)$ we have $\st(\alpha) = \st(\beta)$.
We define a surjection $\uppi \colon D \to \Q_p$ by declaring $\uppi( B(\alpha,\delta) ) = \st(\alpha)$ for all $\alpha \in V$ and also let $\uppi \colon D^n \to \Q_p^n$ be the function giving by applying $\uppi$ coordinate-wise for all $n\ge 2$.
Note that $\uppi$ is surjective and $\shk$-definable.
Let $\approx$ be the equivalence relation on $D$ given by declaring $\alpha\approx\beta$ if and only if $\uppi(\alpha)=\uppi(\beta)$.
Then $\approx$ is $\Sh B_\K$-definable so we identify $\Q_p$ with $D/\!\approx$ and consider $\Q_p$ to be a $\Sh B_\K$-definable set of imaginaries.
It now suffices to show that the field operations on $\Q_p$ are definable in $\Sh B_\K$.
Fix a polynomial $f\in\Z[x_1,\ldots,x_n]$ and let $X, X^*$ be as in the claim.
It suffices to show that $X$ is $\Sh B_\K$-definable.
We  show that $X$ is the image under $\uppi$ of an $\Sh B_\K$-definable subset of $D^n$.
Let $X^*\subseteq \K^n$ be the vanishing set of $f$ over $\K$.
Let $Y$ be the set of $(B_1,\ldots,B_n) \in D^n$ such that $B_1\times\cdots\times B_n$ intersects $X^*\cap V$.
Observe that $Y$ is definable in $\Sh B_\K$ as it is definable in $\shk$.
Note that $Y$ is the set of tuples of the form $(B(\alpha_1,\delta),\ldots,B(\alpha_n,\delta))$ for $(\alpha_1,\ldots,\alpha_n) \in X^*\cap V$.
For any such tuple we have
\[
\uppi\left(B(\alpha_1,\delta),\ldots,B(\alpha_n,\delta)\right) = (\st(\alpha_1),\ldots,\st(\alpha_n)).
\]
Hence $\uppi(Y) = \st(X^*\cap V^n)$.
By the claim $\st(X^*\cap V^n) = X$.
\end{proof}

\section{Trace definability and local trace definability}\label{section:bckgrnd}
We recall trace definability, introduce local trace definability, and prove a number of useful general results about both.

\medskip
Fact~\ref{fact:shelah} shows that if $\Sa M$ is $\nip$ and $\Sa M \to \Sa N$ is an elementary embedding with $\Sa N$ highly saturated then a subset of $M^m$ is $\Sh M$-definable if and only if it is a preimage of an $\Sa N$-definable set under the natural map $M^m \to N^m$.
This allows one to show that many model-theoretic properties pass from $\Sa M$ to $\Sh M$.
We generalize this idea.
Given a map $\uptau\colon M \to N$ we abuse notation by letting $\uptau$ denote the map $M^n \to N^n$ given by applying $\uptau$ componentwise for each $n \ge 2$.
A map $\uptau \colon \Sa M \to \Sa N$ between arbitrary structures is a {\bf trace embedding} if every definable subset of every $M^m$ is of the form $\uptau^{-1}(Y)$ for definable $Y \subseteq N^m$.
Consideration of the graph of equality shows that any trace embedding is injective.
A {\bf trace definition} $\Sa M \rightsquigarrow \Sa N$ is a map $\uptau\colon M \to N^n$ for some $n$ such that every $\Sa M$-definable subset of every $M^m$ is of the form $\uptau^{-1}(Y)$ for $\Sa N$-definable $Y \subseteq N^{mn}$.
We say that $\Sa N$ trace defines $\Sa M$ if there is a trace definition $\Sa M \rightsquigarrow \Sa N$.
It is easy to see that trace definitions form a category and trace definability is therefore transitive.

\begin{fact}\label{fact:embed}
Elementary embeddings are trace embeddings.
If an $L$-structure $\Sa M$ admits quantifier elimination then any embedding of $\Sa M$ into another $L$-structure is a trace embedding.
\end{fact}

Fact~\ref{fact:embed} is \cite[Prop.~1.2]{trace1}.
We now give several equivalent definitions of trace definability.

\begin{fact}\label{fact:trace-def}
The following are equivalent for arbitrary structures $\Sa M$ and $\Sa N$.
\begin{enumerate}[leftmargin=*]
\item $\Sa M$ is trace definable in $\Sa N$.
\item Up to isomorphism $M \subseteq N^n$ for some $n$ and every $\Sa M$-definable subset of every $M^m$ is of the form $Y \cap M^m$ for $\Sa N$-definable $Y \subseteq N^{mn}$.
\item There is a finite set $\Cal E$ of functions $M \to N$ such that every $\Sa M$-definable set is of the form
\[
\{ (a_1,\ldots,a_m) \in M^m : (f_1(a_{i_1}),\ldots,f_n(a_{i_n})) \in Y\}
\]
for some $f_1,\ldots,f_n \in \Cal E$, $i_1,\ldots,i_n \in \{1,\ldots,m\}$, and $\Sa N$-definable $Y \subseteq N^n$.
\item If $L$ is a relational language then every $\Sa M$-definable $L$-structure embeds into an $\Sa N$-definable $L$-structure.
\end{enumerate}
\end{fact}

\begin{proof}
The equivalence of (1) and (2) is clear from the definitions as any trace definition is injective.
If (2) holds then the collection of maps $M \to N$ given by restricting coordinate projections $N^n \to N$ satisfies the condition of (3).
Conversely if $\Cal E = \{f_1, \ldots, f_n\}$ satisfies the condition of (3) then the map $M \to N^n$ given by $a \mapsto (f_1(a), \ldots, f_n(a))$ is easily seen to be a trace definition $\Sa M \rightsquigarrow \Sa N$.
Suppose that (4) holds.
After possibly Morleyizing we may suppose that $\Sa M$ admits quantifier elimination and the language $L$ of $\Sa M$ is relational.
Then $\Sa M$ embeds into an $\Sa N$-definable $L$-structure.
This embedding give a trace definition $\Sa M \rightsquigarrow \Sa N$ by Fact~\ref{fact:embed}.
We leave it to the reader to show that (1) implies (4).
\end{proof}

We say that $\Sa N$ {\bf locally trace defines} $\Sa M$ if there is a possibly infinite collection $\Cal E$ of functions $M \to N$ such that every $\Sa M$-definable set is of the form
\[
\{ (a_1,\ldots,a_m) \in M^m : (f_1(a_{i_1}),\ldots,f_n(a_{i_n})) \in Y\}
\]
for some $f_1,\ldots,f_n \in \Cal E$, $i_1,\ldots,i_n \in \{1,\ldots,m\}$, and $\Sa N$-definable $Y \subseteq N^n$.
We say that $\Cal E$ witnesses local trace definability of $\Sa M$ in $\Sa N$.
Note that if $\Cal E$ witnesses local trace definability of $\Sa M$ in $\Sa N$ and $\Cal F$ witnesses local trace definability of $\Sa N$ in $\Sa O$ then $\{ f \circ e : f \in \Cal F, e \in \Cal E\}$ witnesses local trace definability of $\Sa M$ in $\Sa O$.
Hence local trace definability is also transitive.

\medskip
We think of the $f \in \Cal E$ as $\Sa N$-valued invariants of elements of $\Sa M$, so the idea is that any definable relation between elements of $\Sa M$ reduces to an $\Sa N$-definable relation between the invariants.

\medskip
We now extend these definitions to theories in the natural way.
Let $T$ be a theory.
Then $T$ (locally) trace defines a structure if and only if some model of $T$ does.
Furthermore $T$ (locally) trace defines another theory $T^*$ if some model of $T$ (locally) trace defines a model of $T^*$.
Two theories are (locally) trace equivalent when each (locally) trace defines the other and two structures are (locally) trace equivalent when their theories are.

\begin{proposition}\label{prop:loc trace theories}
If $T$ and $T^*$ are theories then some model of $T^*$ is locally trace definable in a model of $T$ if and only if every model of $T^*$ is locally trace definable in a model of $T$.
\end{proposition}

We proved the analogous statement for trace definability in \cite[Lemma~1.7]{trace1}.
The proof of the local case follows by slight modifications of that argument and is therefore left to the reader.
It follows that local trace definability between theories is transitive and hence that local trace equivalence is an equivalence relation.
Two structures are locally trace equivalent if and only if each is locally trace definable in an elementary extension of the other.

\medskip
We give a number of equivalent definitions of local trace definability.
Let $\Sa M$ be a structure.
Let $I$ be an arbitrary index set and let $nI$ be the disjoint union of $n$  copies of $I$ for every $n \ge 1$.
We identify $M^{nI}$ with $(M^I)^n$ for each $n \ge 1$.
Let $x$ be a tuple $(x_i)_{i \in I}$ of variables and let $\varphi(x)$ range over formulas with parameters from $M$ and variables among the $x_i$.
We say that a subset $X$ of $M^I$ is $\Sa M$-definable if there is $\varphi(x)$ such that $X = \{ a \in M^I : \Sa M \models \varphi(a)\}$.
Of course when $I$ is finite this is just the usual notion of a definable set.

\begin{proposition}\label{prop:loc equiv}
The following are equivalent for any structures $\Sa M$ and $\Sa N$.
\begin{enumerate}[leftmargin=*]
\item $\Sa M$ is locally trace definable in $\Sa N$.
\item Up to isomorphism $M$ is a subset of $N^I$ for some index set $I$ and every $\Sa M$-definable subset of every $M^m$ is of the form $Y \cap M^m$ for some $\Sa N$-definable $Y \subseteq N^{mI}$.
\item There is an index set $I$ and an injection $\uptau\colon M \to N^I$ so that for every $\Sa M$-definable subset $X$ of every $M^m$ there is a $\Sa N$-definable subset $Y$ of $N^{mI}$ such that we have $\alpha \in X$ if and only if $\uptau(\alpha) \in Y$ for all $\alpha \in M^n$.
\item For every $m_1,\ldots,m_k$ and $\Sa M$-definable sets $X_1\subseteq M^{m_1},\ldots,X_k\subseteq M^{m_k}$ there is an injection $\uptau\colon O \to N^n$ for some $n$ and $\Sa N$-definable sets $Y_1\subseteq N^{nm_1},\ldots,Y_k\subseteq N^{nm_k}$ such that we have $\alpha\in X_i$ if and only if $\uptau(\alpha)\in Y_i$ for every $i = 1,\ldots,k$ and $\alpha \in M^{m_i}$.
\item If $L$ is a finite relational language then any $L$-structure which is definable in $\Sa M$ embeds into an $L$-structure which is definable in $\Sa N$.
\item For any $\Sa M$-definable $X\subseteq M^m$ with $m \ge 2$ there is a function $\uptau\colon M \to N^n$ for some $n$ and $\Sa N$-definable $Y\subseteq N^{nm}$ such that we have $\alpha \in X$ if and only if $\uptau(\alpha) \in Y$ for all $\alpha \in M^m$.
\item If $L$ is a language consisting of a single relation of arity $\ge 2$ then every $\Sa M$-definable $L$-structure embeds into an $\Sa N$-definable $L$-structure.
\item Same as (5) but with ``definable" replaced with ``zero-definable".
\item Same as (6) but with ``definable" replaced with ``zero-definable".
\end{enumerate}
\end{proposition}

Here (4) is the motivation of the term ``local trace definability"; local in the sense of handling finitely many definable sets at a time.

\begin{proof}
It is easy to see that (2) is equivalent to (3), that (4) is equivalent to (5), and that (6) is equivalent to (7).
Note that if $\Cal E$ is a collection of functions $M \to N$ witnessing local trace definability of $\Sa M$ in $\Sa N$ then the map $M \to N^{\Cal E}$ given by $\alpha \mapsto (f(\alpha))_{f \in \Cal E}$ satisfies the condition of (3).
Likewise if (2) holds then the collection of functions $M \to N$ given by restricting coordinate projections $N^I \to N$ witnesses local trace definability of $\Sa M$ in $\Sa N$.
It is clear that (5) implies (7).
The equivalence of (5) and (8), and the equivalence of (6) and (9) both follow by Lemma~\ref{lem:local 1} below.
Finally, Lemma~\ref{lem:nfusion} below shows that (7) implies (5).
\end{proof}

\begin{lemma}
\label{lem:fusion}
Let $\Sa M$ be an arbitrary structure,  $L_1,\ldots,L_n$ be  relational languages and  $L$ be the disjoint union of the $L_i$.
Let $\Sa O$ be an $L$-structure and  $\Sa O_i$ be the $L_i$-reduct of $\Sa O$ for each $i = 1,\ldots,n$.
Suppose that each $\Sa O_i$ embeds into an $\Sa M$-definable $L_i$-structure.
Then $\Sa O$ embeds into an $\Sa M$-definable $L$-structure.
\end{lemma}

\begin{proof}
For each $i=1,\ldots,n$ let $\Sa P_i$ be an $\Sa M$-definable $L_i$-structure, $\uptau_i$ be an embedding $\Sa O_i\to\Sa P_i$, and $\uppi_i$ be the coordinate projection $P := P_1\times\cdots\times P_n \to P_i$.
Let $\Sa P$ be the $L$-structure on $P$ given by declaring
\[ \Sa P\models R(\beta_1,\ldots,\beta_k)\quad\Longleftrightarrow\quad \Sa P_i\models R(\uppi_i(\beta_1),\ldots,\uppi_i(\beta_k))\]
for all $i=1,\ldots,n$ and $k$-ary $R\in L_i$.
Then $\Sa P$ is $\Sa M$-definable.
Let $\uptau\colon O \to P$ be given by $\uptau(\alpha)=(\uptau_1(\alpha),\ldots,\uptau_n(\alpha))$ for all $\alpha \in O$.
For any $k$-ary $R \in L_i$ and $\alpha_1,\ldots,\alpha_k \in O$ we have
\begin{align*}
\Sa O \models R(\alpha_1,\ldots,\alpha_k) \quad&\Longleftrightarrow\quad \Sa O_i \models R(\alpha_1,\ldots,\alpha_k) \\ &\Longleftrightarrow\quad \Sa P_i \models R(\uptau_i(\alpha_1),\ldots,\uptau_i(\alpha_k)) \\ &\Longleftrightarrow\quad \Sa P_i\models R( \uppi_i(\uptau(\alpha_1)),\ldots,\uppi_i(\uptau(\alpha_k)) \\ &\Longleftrightarrow\quad \Sa P_{\hspace{.1cm}}\models R(\uptau(\alpha_1),\ldots,\uptau(\alpha_k)).
\end{align*}
Hence $\uptau$ is an embedding $\Sa O\to \Sa P$.
\end{proof}

\begin{lemma}\label{lem:nfusion}
Let $\Sa O = (O; U_1,\ldots,U_n, R_1,\ldots, R_m)$ be a relational structure and $\Sa M$ be an arbitrary structure.
Suppose that each $U_i$ is unary, each $R_i$ has arity $\ge 2$, and that each $(O; R_i)$ embeds into an $\Sa M$-definable structure.
Then $\Sa O$ embeds into an $\Sa M$-definable structure.
\end{lemma}

\begin{proof}
By Lemma~\ref{lem:fusion} it is enough to show that each $(O; U_i)$ embeds into an $\Sa M$-definable structure.
We have $|M| \ge |O|$ as any $(O; R_i)$ embeds into an $\Sa M$-definable structure.
Fix distinct $p, q \in M$.
Then each $(O; U_i)$ embeds into the structure with domain $M \times \{p, q\}$ and a unary relation defining $M \times \{p\}$.
\end{proof}

\begin{lemma}
\label{lem:local 1}
Let $\Sa O$ be a structure and $L$ be a finite relational language.
Any $\Sa O$-definable $L$-structure embeds into an $L$-structure which is zero-definable in $\Sa O$.
\end{lemma}

We leave the proof of Lemma~\ref{lem:local 1} to the reader.
We now give definitions of trace definability and local trace definability in the case when we have quantifier elimination.
This definition can also be applied to the general case via Morleyization.

\begin{lemma}\label{lem:qe case}
Suppose that $\Sa N$ is a structure which admits quantifier elimination and $\Sa M$ is an arbitrary structure.
Then $\Sa N$ locally trace defines $\Sa M$ if and only if there is a collection $\Cal E$ of functions $M \to N$ such that every $\Sa M$-definable set is quantifier free definable in the two-sorted structure $(\Sa N, M, \Cal E)$.
Furthermore $\Sa N$ trace defines $\Sa M$ if we may additionally take $\Cal E$ to be finite.
\end{lemma}

Lemma~\ref{lem:qe case} is easy and left to the reader.
We next give a different characterization of local trace definability which requires the following lemma.
A {\bf $k$-hypergraph} is a set equipped with a symmetric $k$-ary relation $E$ such that $E(a_1,\ldots,a_k)$ implies that the $a_i$ are distinct.

\begin{lemma}\label{lem:local 2}
Suppose that ${\bf e} \colon (V;E) \to (V^*;F)$ is an embedding of a $k$-hypergraph into a $k$-ary relation and suppose that $(V^*;F)$ is definable in a structure $\Sa M$.
Then $(V;E)$ embeds into a $k$-hypergraph that is zero-definable in $\Sa M$.
\end{lemma}

\begin{proof}
By Lemma~\ref{lem:local 1} we may suppose that $(V^*; F)$ is zero-definable.
Let $E^*$ be the $k$-hypergraph on $V^*$ given by declaring $E^*(v_1,\ldots,v_n)$ when the $v_i$ are distinct and we have $F(v_{\sigma(1)},\ldots,v_{\sigma(n)})$ for some permutation $\sigma$ of $\{1,\ldots,n\}$.
Then $(V^*; E^*)$ is an $\Sa N$-definable $k$-hypergraph and ${\bf e}$ gives an embedding $(V;E) \to (V^*;E^*)$.
\end{proof}

In particular any $\Sa M$-definable $k$-hypergraph embeds into a $k$-hypergraph which is zero-definable in $\Sa M$.

\begin{proposition}\label{prop:hy}
The following are equivalent for any structures $\Sa M$ and $\Sa N$.
\begin{enumerate}[leftmargin=*]
\item $\Sa N$ locally trace defines $\Sa M$.
\item Any $\Sa M$-definable $k$-hypergraph embeds into an $\Sa N$-definable $k$-hypergraph for any $k \ge 2$.
\item For any $k \ge 2$, any $k$-hypergraph that is zero-definable in $\Sa M$ embeds into a $k$-hypergraph that is zero-definable in $\Sa N$.
\end{enumerate}
\end{proposition}

\begin{proof}
Lemma~\ref{lem:local 2} shows that (2) and (3) are equivalent.
Suppose that (1) holds and let $(V;E)$ be an $\Sa M$-definable $k$-hypergraph.
Then there is an embedding ${\bf e}$ of $(V;E)$ into an $\Sa N$-definable $k$-ary relation $(V^*; F)$ by Proposition~\ref{prop:loc equiv}.
By Lemma~\ref{lem:local 2} $(V;E)$ embeds into an $\Sa N$-definable $k$-hypergraph.
Hence (1) implies (2).
It remains to show that (2) implies (1).
We need the following claim.
Let $\uppi_i \colon V^k \to V$ be the projection onto the $i$th coordinate for $i = 1,\ldots,k$.

\begin{Claim*}
Fix $k \ge 2$ and a structure $\Sa O$.
For any $k$-hypergraph $(V;E)$ let $R_E$ be the $k$-ary relation on $V^k$ given by  declaring $$R_{E}(v_1,\ldots,v_k)\quad \Longleftrightarrow\quad E(\uppi_1(v_1),\uppi_2(v_2),\ldots,\uppi_k(v_k))$$
for all $v_1,\ldots,v_k\in V^k$.
Then we have the following:
\begin{enumerate}
[leftmargin=*]
\item Any $k$-hypergraph embedding $(V;E)\to(W;F)$ induces an embedding $(V^k; R_E) \to (W^k; R_F)$.
\item If $(V;E)$ is an $\Sa O$-definable $k$-hypergraph then $R_E$ is also $\Sa O$-definable.
\item Any $\Sa O$-definable $k$-ary relation embeds into $(Y^k;R_F)$ for some $\Sa O$-definable $k$-hypergraph $(Y;F)$.
\end{enumerate}
\end{Claim*}

\begin{claimproof}
(1) and (2) are clear from the definition.
We prove (3).
Fix an $\Sa O$-definable $k$-ary relation $R$ on an $\Sa O$-definable set $X$.
After possibly replacing $\Sa O$ with an isomorphic structure suppose that $1,\ldots,k\in O$.
Let $Y = X \times \{1,\ldots,k\}$.
We define a $k$-hypergraph $F$ on $Y$.
Let $(\alpha_1,\imag_1),\ldots,(\alpha_k,\imag_k)$ range over elements of $Y$.
We declare $F((\alpha_1,\imag_1),\ldots,(\alpha_k,\imag_k))$ when $\{\imag_1,\ldots,\imag_k\} = \{1,\ldots,k\}$ and $R(\alpha_{\sigma(1)},\ldots,\alpha_{\sigma(k)})$ where $\sigma$ is the unique permutation of $\{1,\ldots,k\}$ with $\imag_{\sigma(1)} = 1,\ldots,\imag_{\sigma(k)} = k$.
We now give an embedding $\mathbf{e} \colon (X;R) \to (Y^k;R_F)$.
Let $\mathbf{e} \colon X \to Y^k$ be given by $\mathbf{e}(\alpha) = ((\alpha,1),\ldots,(\alpha,k))$.
Suppose that $\alpha_1,\ldots,\alpha_k \in X$.
Then
\begin{align*}
R(\alpha_1,\ldots,\alpha_k) &\quad\Longleftrightarrow \quad F((\alpha_1,1),(\alpha_2,2),\ldots,(\alpha_k,k)) \\
&\quad\Longleftrightarrow \quad R_{F}(\mathbf{e}(\alpha_1),\ldots,\mathbf{e}(\alpha_k))
\end{align*}
Note that $F$ and $\mathbf{e}$ are both $\Sa O$-definable. 
\end{claimproof}

We now suppose that (2) holds and show that $\Sa N$ locally trace defines $\Sa M$.
By Proposition~\ref{prop:loc equiv} it is enough to fix an $\Sa M$-definable $k$-ary relation $(V;R)$ with $k \ge 2$ and show that $(V;R)$ embeds into an $\Sa N$-definable $k$-ary relation.
By the claim $(V;R)$ embeds into $(W^k; R_E)$ for some $\Sa M$-definable $k$-hypergraph $(W; E)$.
By assumption $(W; E)$ embeds into an $\Sa N$-definable $k$-hypergraph $(Y; F)$.
Any embedding $(W; E) \to (Y; F)$ induces an embedding $(W^k; R_E) \to (Y^k; R_F)$ by the claim, hence $(V;R)$ embeds into $(Y^k; R_F)$.
\end{proof}

We now consider structures admitting quantifier elimination in relational languages.

\begin{proposition}\label{prop:qe}
Suppose that  $L$ is a relational language and $\Sa M$ is an $L$-structure with quantifier elimination.
Then an arbitrary structure $\Sa N$ locally trace defines $\Sa M$ if and only if $(M;R)$ embeds into an $\Sa N$-definable structure for every $R \in L$ of arity $\ge 2$.
If $L$ is finite then $\Sa N$ trace defines $\Sa M$ if and only if $\Sa N$ locally trace defines $\Sa M$.
\end{proposition}

It follow in particular that if $\Sa O$ is finitely homogeneous then a structure $\Sa M$ trace defines $\Sa O$ if and only if it locally trace defines $\Sa O$.

\begin{proof}
The second claim follows from the first claim, Lemma~\ref{lem:fusion}, and Fact~\ref{fact:embed}.
We prove the first claim.
The left to right implication follows by Proposition~\ref{prop:loc equiv}.
Suppose that the right hand side holds.
By Proposition~\ref{prop:loc equiv} it is enough to suppose that $X$ is an $\Sa M$-definable subset of $M^k$ and show that $(M; X)$ embeds into an $\Sa N$-definable $k$-ary relation.
Suppose that $X$ is quantifier-free definable in $(M;R_1,\ldots,R_n)$ for $R_1,\ldots,R_n \in L$ and let $\varphi(x_1,\ldots,x_k)$ be a quantifier-free $\{ R_1, \ldots, R_n\}$-formula defining $X$.
By Lemma~\ref{lem:nfusion} $(M;R_1,\ldots,R_n)$ embeds into an $\Sa N$-definable structure $\Sa P$.
So we may suppose that $(M;R_1,\ldots,R_n)$ is a substructure of $\Sa P$.
Let $Y$ be the subset of $P^k$ defined by $\varphi$.
Then the inclusion $M \to P$ gives an embedding $(M;X) \to (P;Y)$.
\end{proof}

We now give a somewhat technical lemma.

\begin{lemma}\label{lem:tech}
Suppose that $L^*$ is an expansion of a language $L$ by relations and $\Sa M$ is an $L^*$-structure with quantifier elimination.
Let $\Sa N$ be an arbitrary structure and suppose that for any $k$-ary $R \in L^* \setminus L$ there is an $\Sa N$-definable $L$-structure $\Sa P$, an embedding $\uptau$ of the $L$-reduct of $\Sa M$ into $\Sa P$, and an $\Sa N$-definable $X \subseteq P^k$ such that we have $\Sa M \models R(\alpha)$ if and only if $\uptau(\alpha) \in X$ for all $\alpha \in M^k$.
Then $\Sa N$ locally trace defines $\Sa M$.
\end{lemma}

The proof of Lemma~\ref{lem:tech} is similar to that of Proposition~\ref{prop:qe} and is left to the reader.

\medskip
The following characterizations of $\nip$ and $k$-$\nip$ in terms of trace definability are given in \cite[Prop.~2.5]{trace1}.
(The generic $k$-hypergraph is the \Fraisse limit of the class of finite $k$-hypergraphs.)
The characterizations in terms of local trace definability follow by this and the second claim of Proposition~\ref{prop:qe}.

\begin{fact}\label{fact:nip char}
A theory is $k$-$\nip$ for some $k \ge 1$ if and only if it does not (locally) trace define the generic $(k+1)$-ary hypergraph.
Furthermore a theory is $\nip$  if and only if it does not (locally) trace define the \Fraisse limit of the class of finite bipartite graphs.
\end{fact}

Hence $k$-$\nip$ is preserved under local trace definability for every $k \ge 1$.
(It is also easy to see this directly from the definition.)
It follows that generic hypergraphs of different airities are distinct modulo local trace equivalence and hence there are infinitely many structures modulo local trace equivalence.
We recall several more preservation results from~\cite{trace1}.

\begin{fact}\label{fact:preserve}
The following properties are preserved under trace definability: superstability, total transcendence, finiteness of Morley rank, finiteness of dp-rank, finiteness of U-rank.
\end{fact}

We now consider theories that locally trace define every structure.
We say that a theory $T$ is $\infty$-$\ip$ if it is $k$-$\ip$ for every $k\ge 1$ and is $\infty$-$\nip$ if it is not $\infty$-$\ip$.
Let $\Sa H_k$ be the generic $k$-hypergraph for each $k \ge 2$.

\begin{proposition}\label{prop:loctrmax}
The following are equivalent for any theory $T$.
\begin{enumerate}[leftmargin=*]
\item $T$ locally trace defines every structure.
\item $T$ is $\infty$-$\ip$.
\item $T$ trace defines $\Sa H_k$ for every $k \ge 2$.
\end{enumerate}
\end{proposition}

\begin{proof}
Fact~\ref{fact:nip char} shows that (2) and (3) are equivalent.
Proposition~\ref{prop:qe} shows that (1) implies (3).
Suppose that (3) holds.
We show that $T$ locally trace defines every structure.
By Proposition~\ref{prop:loc equiv} it is enough to fix $k \ge 2$ and show that any $k$-ary relation embeds into a $k$-ary relation that is definable in a model of $T$.
It is enough to show that any $k$-ary relation embeds into a $k$-ary relation that is definable in an elementary extension of $\Sa H_k$.
This follows by the proof of Proposition~\ref{prop:hy} and the fact that any $k$-hypergraph embeds into an elementary extension of $\Sa H_k$.
\end{proof}

\subsection{Trace definability and the Shelah completion}\label{section:shcomp}
We now consider the relationship between trace definability and the Shelah completion.

\begin{lemma}
\label{lem:she-last}
Let $e$ be an elementary embedding $\Sa M \to \Sa N$ between $\nip$ structures.
Then $e$ gives a trace embedding $\Sh M \to \Sh N$ and if $\Sa N$ is $|M|^+$-saturated then $e$ gives a trace embedding $\Sh M \to \Sa N$.
\end{lemma}

This is immediate from the definitions, Fact~\ref{fact:shelah}, and Fact~\ref{fact:ksh}.
Proposition~\ref{prop:she-0} now follows as $\Sa M$ is a reduct of $\Sh M$.

\begin{proposition}
\label{prop:she-0}
Every $\nip$ structure is trace equivalent to its Shelah completion.
\end{proposition}

Corollary~\ref{cor:poizat-b} follows from Proposition~\ref{prop:she-0} and external definability of convex sets.

\begin{corollary}
\label{cor:poizat-b}
If $\Sa M$ is a $\nip$ expansion of a linear order then any expansion of $\Sa M$ by unary relations defining convex subsets of $M$ is trace equivalent to $\Sa M$.
\end{corollary}

We now reconsider the examples given in Section~\ref{section:new}.

\begin{proposition}\label{prop:shn}
Let $\Sa H_\lambda$, $\Sa P$, $\Sa R$, and $\Sa B$ be as in Section~\ref{section:new}.
Then $\Sa H_\lambda$ is trace equivalent to $\rfield$, $\Sa P$ is trace equivalent to $\Sa R$, and $\Sa B$ is trace equivalent to $\Q_p$.
\end{proposition}

\begin{proof}
We showed above that $\Sa R$ is interpretable in $\Sh P$, $\rfield$ is interpretable in the Shelah completion of any $\aleph_1$-saturated elementary extension of $\Sa H_\lambda$, and $\Q_p$ is interpretable in the Shelah completion of a certain elementary extension of $\Sa B$.
It follows from Proposition~\ref{prop:she-0} that $\Sa R$, $\rfield$, $\Q_p$, is trace definable in $\Th(\Sa P)$, $\Th(\Sa H_\lambda)$, $\Th(\Sa B)$, respectively.
Now $\Sa B$ is interpretable in, and hence trace equivalent to, $\Q_p$.
Fact~\ref{fact:dms} shows that $\Sa P$ trace embeds into an elementary extension of $\Sa R$, so $\Sa P$ is trace equivalent to $\Sa R$.
Fact~\ref{fact:vddgun} shows that $h \mapsto \lambda^h$ gives a trace embedding $\Sa H_\lambda \to \rfield$, hence $\Sa H_\lambda$ is trace equivalent to $\rfield$.
\end{proof}

\subsection{Multisorted structures and disjoint unions}\label{section:disjoint}
We briefly consider trace definability and local trace definability between multisorted structures and theories.
We need this to construct infinite disjoint unions of structures.
Let $\Sa N$ and $\Sa M$ be multi-sorted structures with sorts $(N_t)_{t \in S'}$ and $(M_t)_{t \in S}$, respectively.
We only treat the case when $\Sa N$ admits quantifier elimination and extend to the general case via Morleyization.

\medskip
We say that $\Sa M$ is trace definable in $\Sa N$ if for every $s\in S$ there is a finite $X_s\subseteq S'$ and a finite collection $\Cal E_s$ of functions $M_s \to N_t$, $t\in X_s$, such that every $\Sa M$-definable set is quantifier free definable in $(\Sa N, (M_t)_{t \in S'} ,(\Cal E_s)_{s\in S})$.
We say that $\Sa M$ is locally trace definable in $\Sa N$ if for every $s\in S$ there is a collection $\Cal E_s$ of functions $M_s \to N_t$, $t\in S'$, such that every $\Sa M$-definable set is quantifier free definable in $(\Sa N, (M_t)_{t \in S'} ,(\Cal E_s)_{s\in S})$.
It is easy to see that the previously proven results generalize in a suitable form to multi-sorted structures.
We leave this and other basic facts to the reader.

\medskip
Given a family of (possibly multi-sorted) languages $(L_i)_{i \in I}$ let $L_\sqcup$ be the disjoint union of the $L_i$, considered as a multi-sorted language in the natural way.
(So in particular if each $L_i$ is one-sorted the $L_\sqcup$ is $|I|$-sorted.)
Given a family $(\Sa M_i)_{i\in I}$ of structures we let $\bigsqcup_{i\in I} \Sa M_i$ be the disjoint union of the $\Sa M_i$ considered as an $L_\sqcup$-structure in the natural way.

\medskip
Fact~\ref{fact:fv} is immediate.

\begin{fact}\label{fact:fv}
Let $(\Sa M_i)_{i\in I}$ be a family of structures.
Fix $i_1, \ldots, i_k \in I$ and $m_1,\ldots,m_k$.
Then any subset of $M^{m_1}_{i_1} \times \cdots \times M^{m_k}_{i_k}$ that is definable in $\bigsqcup_{i\in I} \Sa M_i$ is of the form $X_1 \times \cdots \times X_k$ where each $X_j \subseteq M^{m_j}_{i_j}$ is definable in $\Sa M_{i_j}$.
\end{fact}

Lemma~\ref{lem:k-dis} is immediate from the definition of (local) trace definability between multisorted structures and Fact~\ref{fact:fv}.

\begin{lemma}
\label{lem:k-dis}
Suppose that $T$ is a theory and $(\Sa M_i)_{i\in I}$ is a family of structures, both possibly multi-sorted.
Then $T$ (locally) trace defines $\bigsqcup_{i\in I} \Sa M_i$ if and only if $T$ (locally) trace defines each $\Sa M_i$.
\end{lemma}

Lemma~\ref{lem:disjoint union} is immediate from Lemma~\ref{lem:k-dis}.

\begin{lemma}\label{lem:disjoint union}
Let $(\Sa M_i)_{i \in I}$ and $(\Sa M^*_i)_{i \in I}$ be  families of (possibly multisorted) structures.
Suppose that $\Sa M_i$ is (locally) trace equivalent to $\Sa M^*_i$ for each $i \in I$.
Then $\bigsqcup_{i\in I}\Sa M_i$ is (locally) trace equivalent to $\bigsqcup_{i\in I}\Sa M^*_i$.
\end{lemma}

\section{Zarankiewicz bounds}\label{section:zaran}
Consider a bipartite graph $\Sa G$ with sorts $V,W$ and edge relation $E \subseteq V \times W$.
Then $\Sa G$ is {\bf $K_{m,n}$-free} if the complete bipartite graph with sorts $\{1,\ldots,m\}, \{1,\ldots,n\}$ is not a subgraph of $\Sa G$.
We say that a class $\Cal C$ of finite bipartite graphs has near linear Zarankiewicz bounds if for any $m$ and positive $\varepsilon \in \R$ there is positive $\lambda \in \R$ such that any $K_{m,m}$-free member $\Sa G^*$ of $\Cal C$ has at most $\lambda |\Sa G^*|^{1 + \varepsilon}$ edges.
We say that $\Sa G$ has \textbf{near linear Zarankiewicz bounds} if its age does.
Furthermore a structure has near linear Zarankiewicz bounds if every definable bipartite graph has near linear Zarankiewicz bounds and a theory $T$ has near linear Zarankiewicz bounds if every (equivalently: some) $\Sa M \models T$ has near linear Zarankiewicz bounds.
The name\footnote{Feel free to come up with a better name.} comes from Zarankiewicz's problem.
The ``near" is necessary, as even $(\Q;<)$ fails to have linear bounds, see \cite{zaran}.

\medskip
We recall a notion of Chernikov and Mennen~\cite{chernikov-mennen}.
A collection of sets is a $(2,1)$-collection if whenever two members of the collection intersect then one is contained in the other.
A formula $\varphi(x,y)$ in the language of $T$ is a \textbf{(2,1)-semi-equation} if the collection $(\varphi(\alpha,M^{|y|}))_{\alpha\in M^{|x|}}$ is a $(2,1)$-collection for any (equivalently: some) $\Sa M\models T$.
Furthermore $T$ is $(2,1)$-semi-equational if every formula $\varphi(x,y)$ is a boolean combination of $(2,1)$-semi-equations and a structure is $(2,1)$-semi-equational when its theory is.
Fact~\ref{fact:am} is proven in~\cite[Prop.~2.26]{chernikov-mennen}.

\begin{fact}
\label{fact:am}
If $T$ is $(2,1)$-semi-equational then $T$ has near linear Zarankiewicz bounds.
\end{fact}

Lemma~\ref{lem:disjoint zaran} follows by Fact~\ref{fact:fv} and the fact that $\varphi(x, y) \land \varphi^*(x^*, y^*)$ is a $(2, 1)$-semi-equation when $\varphi(x,y)$ and $\varphi^*$ are $(2, 1)$-semi-equations.
We leave the details to the reader.

\begin{lemma}
\label{lem:disjoint zaran}
Let $(\Sa M_i)_{i \in I}$ be a family of structures and suppose that each $\Sa M_i$ is $(2,1)$-semi-equational.
Then the disjoint union $\bigsqcup_{i\in I}\Sa M_i$ is also $(2,1)$-semi-equational.
\end{lemma}

The proof of Lemma~\ref{lem:blankm} is easy and left to the reader.


\begin{lemma}
\label{lem:blankm}
Suppose that $\Sa M$ is locally trace definable in $\Sa N$.
Then every $\Sa M$-definable bipartite graph embeds into a $\Sa N$-definable bipartite graph.
\end{lemma}

Lemma~\ref{lem:blank-2} follows by Lemma~\ref{lem:blankm}.

\begin{lemma}
\label{lem:blank-2}
Near linear Zarankiewicz bounds are preserved under local trace definability.
\end{lemma}

We show that near linear Zarankiewicz bounds is a $\nip$-theoretic property.

\begin{fact}
\label{fact:blank-0}
Any theory with near linear Zarankiewicz bounds is $\nip$.
\end{fact}

As we will see below the class of all finite bipartite graphs does not have near linear Zarankiewicz bounds, see \cite{erdos-zaran} for the original proof.
Hence the \Fraisse limit of the class of finite bipartite graphs does not have near linear Zarankiewicz bounds.
Hence Fact~\ref{fact:blank-0} follows from Fact~\ref{fact:nip char}.
We next show that certain rings, in particular infinite fields, do not have near linear Zarankiewicz bounds.
We first gather some tools.
All rings are unitary.
A {\bf domain} is a (possibly noncommutative) ring without zero divisors.

\begin{fact}\label{fact:jacobsen}
Any infinite domain interprets an infinite field of the same characteristic.
\end{fact}

This is essentially proven in \cite[Prop.~2.3]{distal-valued}.
We recount their argument.
We apply the following theorem of Jacobsen: if $R$ is ring such that every $r \in R$ satisfies $r^n = r$ for some $n \ge 2$ then $R$ is commutative~\cite[Thm.~12.10]{Lam}.

\begin{proof}
Let $R$ be an infinite domain.
If $R$ is commutative then $R$ interprets its field of fractions.
Suppose $R$ is not commutative.
It is enough to produce an infinite definable commutative subring of $R$.
By Jacobsen's theorem there is $r \in R$ such that $r^n \ne r$ for every $n \ge 2$.
As $R$ is a domain $r$ is not nilpotent, so the powers of $r$ are distinct.
Let $C$ be the set of $s \in R$ which commute with $r$, so $C$ is a definable subring of $R$.
Let $Z$ be the center of $C$, so $Z$ is a definable commutative subring of $R$.
Finally, each $r^n$ is in $Z$, so $Z$ is infinite.
\end{proof}

We let $\F_p$ be the field with $p$ elements.

\begin{proposition}\label{prop;p}
Fix a prime $p$.
If $E$ is an infinite domain of characteristic $p$ then $\Th(E)$ locally trace defines the theory of the algebraic closure of $\F_p$.
\end{proposition}

We apply the following results of Hempel~\cite[Thm.~6.3 and 7.3]{hempel-field}.
\begin{enumerate}[leftmargin=*]
\item If $E$ is a field which has a non-separably closed $\mathrm{PAC}$ subfield $F$ such that $F$  is algebraically closed in $E$, then $E$ is $\infty$-$\ip$.
\item An infinite field which is not Artin-Schreier closed is $\infty$-$\ip$.
\end{enumerate}

\begin{proof}
By Fact~\ref{fact:jacobsen} we may suppose that $E$ is an infinite field of characteristic $p$.
Let $E_\mathrm{alg}$ be the algebraic closure of $\F_p$ in $E$.
If $E_\mathrm{alg}$ is finite then $E_\mathrm{alg}$ has an Artin-Schreier extension, hence $E$ has an Artin-Schreier extension, hence $E$ is $\infty$-$\ip$ by (2) above, and so $\Th(E)$ locally trace defines every structure by Proposition~\ref{prop:loctrmax}.
If $E_\mathrm{alg}$ is algebraically closed then  the inclusion $E_\mathrm{alg}\to E$ is a trace embedding by quantifier elimination for algebraically closed fields.
Suppose that $E_\mathrm{alg}$ is infinite and not algebraically closed.
An infinite algebraic extension of a finite field is $\pac$~\cite[Cor.~11.2.1]{field-arithmetic}.
Hence $E$ is $\infty$-$\ip$ by (1) above.
\end{proof}

\begin{proposition}
\label{prop:field-blank}
Suppose that $R$ is a ring.
If either $R$ has characteristic zero or $R$ is an infinite domain then $R$ does not have near linear Zarankiewicz bounds.
\end{proposition}

This is sharp in that there are infinite positive characteristic rings with near linear Zarankiewicz bounds.
Fix an infinite $\F_p$-vector space $V$.
Let $R$ be the ring with domain $\F_p \times V$, componentwise addition, and multiplication given by
\[
(\lambda, v) \cdot (\lambda^*, v^*) = (\lambda \lambda^*, \lambda v + \lambda^* v^*).
\]
Then $R$ is interpretable in $V$ and hence has near linear Zarankiewicz bounds by Proposition~\ref{prop:am} below.

\begin{proof}
Let $\Sa G$ be the bipartite graph given by the incidence relation between points and non-vertical lines in $R^2$, so $\Sa G = (R^2, R^2; E)$ where  we have $E((a,b), (c,d))$ when $b = ca + d$.
Note that if $D$ is a subring of $R$ which is a domain then the substructure of $\Sa G$ with sorts $D^2, D^2$ is $K_{2,2}$-free.

\medskip
We first treat the case when $R$ is of characteristic zero.
We just need to recall the usual witness for sharpness in the lower bounds of Szemeredi-Trotter.
We fix $n \ge 1$, declare $V = \{1,\ldots,n\} \times \{1,\ldots,2n^2\}$ and let $W = \{1,\ldots,n\} \times \{1,\ldots,n^2\}$.
(So $W$ is the set of lines with slope in $\{1,\ldots,n\}$ and $y$-intercept in $\{1,\ldots,n^2\}$.)
Then $(V, W; E)$ is a $K_{2,2}$-free subgraph of $\Sa G$.
Furthermore $|V + W| = 3n^3$ and there are $n^4$ incidences between $V$ and $W$.
Hence $\Sa G$ does not have near linear Zarankiewicz bounds.

\medskip
Now suppose that $R$ is a characteristic $p$ domain.
Then $\Sa G$ is $K_{2,2}$-free.
By Proposition~\ref{prop;p} and preservation of near linear Zarankiewicz bounds under local trace definability we may suppose that $R$ is the algebraic closure of $\F_p$.
Fix $n \ge 1$ and let $F$ be the subfield of $R$ with $q = p^n$ elements.
Let $V = F^2 = W$.
(So $W$ is the set of non-vertical lines between elements of $F^2$.)
Then $|V| = q^2 = |W|$ and as every $\ell \in W^*$ contains $q$ points in $E^2$ there are $q^3$ incidences between $V$ and $W$.
Hence $\Sa G$ does not have near linear Zarankiewicz bounds.
\end{proof}

We now show that certain structures have near linear Zarankiewicz bounds.
An {\bf abelian structure} is structure which is, up to interdefinability, an abelian group $A$ equipped with a family of subgroups of various $A^n$.
For example any module is an abelian structure.
It is an important fact that an expansion of an abelian group is one-based if and only if it is an abelian structure if and only if  every definable subset of $A^n$ is a boolean combination of cosets of definable subgroups of $A^n$~\cite{HP-weakly-normal}.

\begin{proposition}
\label{prop:am}
Any abelian structure is $(2,1)$-semi-equational and hence has near linear Zarankiewicz bounds.
\end{proposition}

\begin{proof}
Let $\Sa H$ be an an abelian structure.
By Fact~\ref{fact:am} it suffices to show that $\Sa H$ is $(2,1)$-semi-equational.
Let $\varphi(x,y)$ be a formula with  $|x|=m,|y|=n$.
Then $\varphi(x,y)$ is equivalent to a finite boolean combination of  formulas  which define cosets of subgroups of $H^{m+n}$.
Suppose that $\varphi(x,y)$ defines a coset of a subgroup $J$ of $H^{m+n}$.
Then there is a subgroup $J'$ of $H^{m + n}$ such that $J_\alpha$ is a coset of $J'$ for every $\alpha \in H^m$.
Hence $(J_\alpha)_{\alpha \in H^m}$ is a $(2, 1)$-collection and so $\varphi(x,y)$ is a $(2,1)$-semi-equation.
\end{proof}

\begin{fact}
\label{fact:mct}
Any ordered vector space over an ordered division ring is $(2,1)$-semi-equational and hence has near linear Zarankiewicz bounds.
\end{fact}


It follows from quantifier elimination that ordered vector spaces are $(2, 1)$-semi-equational, see \cite[Prop.~3.6]{chernikov-mennen}.
The second claim follows from the first by Fact~\ref{fact:am}.
(It was previously shown in  \cite[Thm.~C]{zaran} that $\rgoup$ has near linear Zarankiewicz bounds.)
Corollary~\ref{cor:o-min} follows as an o-minimal expansion of an ordered abelian group that does not interpret an infinite field is a reduct of an ordered vector space over an ordered division ring~\cite{PS-Tri}.

\begin{corollary}\label{cor:o-min}
If $\Sa R$ is an o-minimal expansion of an ordered group then $\Th(\Sa R)$ locally trace defines an infinite field if and only if $\Sa R$ interprets an infinite field.
\end{corollary}

Lemma~\ref{lem:vector space}  will be applied to some examples below.

\begin{lemma}
\label{lem:vector space}
A disjoint union of an abelian structure with an ordered vector space is $(2,1)$-semi-equational and hence has near linear Zarankiewicz bounds.
\end{lemma}

\begin{proof}
Apply 
Proposition~\ref{prop:am},
Fact~\ref{fact:mct}, Lemma~\ref{lem:disjoint zaran}, Fact~\ref{fact:am}, and Lemma~\ref{lem:blank-2}.
\end{proof}

\subsection{The strong \Erdos-Hajnal Property}\label{section:eh}
We now consider a property which rules out local trace definability of infinite positive characteristic fields.
Let $\Sa G$ be as in the previous section.
We say that $\Sa G$ has the \textbf{strong \Erdos-Hajnal property} if there is a real number $\delta > 0$ such that for every finite $A \subseteq V, B \subseteq W$ there are $A^* \subseteq A$, $B^* \subseteq B$ such that $|A^*| \geq \delta |A|$, $|B^*| \geq \delta B$, and $A^* \times B^*$ is either contained in or disjoint from $E$.
A structure has the \textbf{strong \Erdos-Hajnal property} if all definable bipartite graphs have the strong \Erdos-Hajnal property and a theory has the strong \Erdos-Hajnal property when all of its models do.

\medskip
As we will see below there are countable bipartite graphs that do not have the strong \Erdos-Hajnal property, so the generic bipartite graph does not have the strong \Erdos-Hajnal property.
Thus by Fact~\ref{fact:nip char} the strong \Erdos-Hajnal property implies $\nip$.

\medskip
Proposition~\ref{prop:eh-trace} is clear from the definitions.

\begin{proposition}
\label{prop:eh-trace}
The strong \Erdos-Hajnal property is preserved under local trace definability.
\end{proposition}

Fact~\ref{fact:distal} is due to Chernikov and Starchenko~\cite[Thm.~1.9]{CS}.

\begin{fact}
\label{fact:distal}
Any distal structure has the strong \Erdos-Hajnal property.
\end{fact}

Fact~\ref{fact:eh} for fields is also due to Chernikov and Starchenko~\cite[Section 6]{CS}.
The general case reduces to the field case by Fact~\ref{fact:jacobsen}.

\begin{fact}
\label{fact:eh}
An infinite domain of positive characteristic does not satisfy the strong \Erdos-Hajnal property.
\end{fact}

We say that a structure is \textbf{predistal} if it has a distal expansion, or equivalently is interpretable in a distal structure.
Proposition~\ref{prop:trace-distal} follows from the previous three results.

\begin{proposition}
\label{prop:trace-distal}
A predistal structure cannot locally trace define an infinite domain of positive characteristic.
\end{proposition}

Again, this is sharp in that there are infinite positive characteristic rings which are interpretable in distal structures.
The ring described below Proposition~\ref{prop:field-blank} is interpretable in the theory of $\F_p$-vector spaces and the theory of $\F_p$-vector spaces has a distal expansion~\cite[\S~2.1]{distal-valued}.
Corollary~\ref{cor:trace-distal} enumerates some special cases of Proposition~\ref{prop:trace-distal}.

\begin{corollary}
\label{cor:trace-distal}
The following structures cannot locally trace define infinite positive characteristic domains.
\begin{enumerate}[leftmargin=*]
\item Any dp-minimal expansion of a linear order.
In particular any o-minimal structure.
\item Any P-minimal structure with definable Skolem functions, in particular any $p$-adically closed field.
\item Any algebraically closed valued field of equicharacteristic zero.
\end{enumerate}
\end{corollary}

\begin{proof}
Dp-minimal expansions of  linear orders are distal~\cite[Example 9.20]{Simon-Book} and a P-minimal structure with definable Skolem functions is distal by \cite[\S~2.3.3]{CS}.
Finally, the theory of equicharacteristic zero algebraically closed valued fields is interpretable in the theory of real closed valued fields in the obvious way, and the theory of real closed valued fields is weakly o-minimal, hence any real closed valued field is a dp-minimal expansion of a linear order.
\end{proof}

\section{Ordered abelian groups, cyclically ordered abelian groups, and ordered vector spaces}\label{section:oag}

We prove a number of results about abelian groups, ordered abelian groups, and ordered vector spaces modulo trace equivalence and local trace equivalence.

\subsection{Vector spaces and abelian groups}
We begin with a result on vector spaces.

\begin{proposition}
\label{prop:lve}
Any  vector space over a division ring is locally trace equivalent to its underlying additive group.
Hence the local trace equivalence type of an infinite vector space over a division ring is determined by the characteristic of the ring.
\end{proposition}

\begin{proof}
The case of a finite vector space is trivial.
Suppose that $\Sa V, \Sa V^*$ are infinite vector spaces over division rings $\E,\E^*$ of the same characteristic, respectively.
Then $\E,\E^*$ have the same prime subfield $\F$, hence the $\F$-vector space reducts of $\Sa V, \Sa V^*$ are elementarily equivalent, hence the underlying additive groups of $\Sa V, \Sa V^*$ are elementarily equivalent.
Hence it is enough to prove the first claim.
Every term in the theory of $\E$-vector spaces is equivalent to a term of the form $\lambda_1 x_1 + \cdots + \lambda_n x_n$ for $\lambda_1,\ldots,\lambda_n \in \E$.
It follows by quantifier elimination for vector spaces that the collection of functions $v \mapsto \lambda v$ for $\lambda\in \E$ witnesses local trace definability of any $\E$-vector space in its underlying additive group.
\end{proof}

\begin{proposition}
\label{prop:ab}
If $\Gamma$ is an infinite exponent group then $\Th(\Gamma)$ trace defines $(\Q;+)$.
\end{proposition}


\begin{proof}
By Fact~\ref{fact:embed} and quantifier elimination for $(\Q; +)$ it is enough to show that $(\Q;+)$ embeds into an elementary extension of $\Gamma$.
This is a compactness exercise.
\end{proof}

We let $\Sa C$ be the structure with domain the Cantor set and a unary relation for each clopen subset.
This is trace equivalent to any set $X$ equipped with $\aleph_0$ relations which form an atomless boolean algebra of subsets of $X$.
Fact~\ref{fact:cantor} is proven in \cite[Prop.~3.3]{trace1}.

\begin{fact}\label{fact:cantor}
An arbitrary theory $T$ trace defines $\Sa C$ if and only if $T$ is not totally transcendental.
Any structure in a countable unary relational language is trace definable in $\Th(\Sa C)$.
\end{fact}

An abelian group $A$ is {\bf non-singular} if the $p$-torsion subgroup of $A$ is finite and $|A/pA| < \infty$ for all primes $p$.
In particular torsion free abelian group is non-singular if and only if $pA$ is a finite index subgroup for all primes $p$ if and only if $nA$ is a finite index subgroup for all $n \ge 1$.
Several people independently showed that an ordered abelian group is dp-minimal if and only if it is non-singular~\cite{Farre,JSW-field,Johnson-thesis}.

\begin{fact}
\label{fact:finite rank}
Any finite rank torsion free abelian group is non-singular.
\end{fact}

Let $B$ be an abelian group.
An embedding $A \to B$  is \textbf{pure} if any element of $A$ is $k$-divisible in $A$ if and only if its image is $k$-divisible in $B$.
A subgroup of $B$ is pure if the inclusion is pure.
If $B$ is torsion free then a subgroup $A$ is pure if and only if $B/A$ is torsion free.


\begin{proof}
Fix a prime $p$.
Let $r$ be the rank of $A$ and $\dim_p(A)$ be the dimension of $A/pA$ as an $\F_p$-vector space.
We show that $\dim_p(A) \le r$ by applying induction on $r$.
The rank one case follows by \cite[Example~9.10]{fuchs}.
Suppose $r \ge 2$.
Fix non-zero $\beta\in A$ and let $A'$ be the set of $\beta^*\in A$ such that $\beta^*=q\beta$ for some $q\in\Q$.
Let $A'' = A/A'$.
Then $A'$ is a pure rank one subgroup of $A$ and $A''$ is torsion free by purity.
By \cite[Thm.~29.1(c)]{Fuchs_vol_1} we have an exact sequence $0\to A'/pA'\to A/pA\to A''/pA''\to 0$.
Hence $\dim_p(A) = \dim_p(A')  + \dim_p(A'')$.  
So by induction we have
$\dim_p(A) \le 1 + (r - 1) = r$.
\end{proof}

Our next goal is to prove the following.

\begin{proposition}\label{prop:Z-decomp}
Any non-divisible non-singular torsion free abelian group is trace equivalent to $(\Q; +) \sqcup \Sa C$.
Hence any such group is also trace equivalent to $(\Z; +)$.
Furthermore any non-singular torsion free abelian group is trace definable in the theory of any non-totally transcendental expansion of an unbounded exponent group.
\end{proposition}

So, from our perspective, $(\Z; +)$ is the simplest structure which at least as complicated as $(\Q; +)$ and is not totally transcendental.
The disjoint union decomposition is non-trivial.
First, $\Th(\Q; +)$ is totally transcendental and hence cannot trace define $\Sa C$ by Fact~\ref{fact:cantor}.
Secondly any unary relational structure is stable, so $\Sa C$ is monadically stable and hence $\Th(\Sa C)$ cannot trace define an infinite group by Fact~\ref{fact:lb} below.
Proposition~\ref{prop:Z-decomp} is also sharp in that one can apply the work of Evans, Pillay, and Poizat~\cite{group-in-a-group} to show that $\Th(\Z^n ; +)$ interprets $\Th(\Z^m ; +)$ if and only if $n$ divides $m$.
Furthermore, Palac\'{\i}n-Sklinos showed that any proper expansion of $(\Z^n;+)$ has infinite U-rank~\cite{palacin-sklinos}.
As $(\Z; +)$ has U-rank one it follows by Fact~\ref{fact:preserve} that $\Th(\Z; +)$ cannot trace define a proper expansion of $(\Z; +)$.

\medskip
We first recall the quantifier elimination for abelian groups.
Let $A$ be an abelian group, written additively.
Given  $k \in \N$, we write $k|\alpha$ when $\alpha = k \beta$ for some $\beta \in A$.
We consider each $k|$ to be a unary relation symbol and let $L_{\mathrm{div}}$ be the expansion of the language of abelian groups by the $k|$.
Note that an embedding of abelian groups is pure if and only if it is an $L_\mathrm{div}$-embedding. 
Fact~\ref{fact:abelian qe} is due to Szmielew, see~\cite[Thm.~A.2.2]{Hodges}.

\begin{fact}
\label{fact:abelian qe}
Any abelian group admits quantifier elimination in $L_\mathrm{div}$.
\end{fact}

\begin{proof}[Proof of Proposition~\ref{prop:Z-decomp}]
The last claim follows from the first claim, Proposition~\ref{prop:ab}, and Fact~\ref{fact:cantor}.
The second claim follows from the first as $(\Z; +)$ is a non-divisible non-singular torsion free abelian group.
We prove the first claim.
Let $A$ be a non-divisible non-singular torsion free abelian group.
We show that $\Th(A)$ trace defines $(\Q; +) \sqcup \Sa C$.
By Lemma~\ref{lem:k-dis} it suffices to show that $(\Q; +)$ and $\Sa C$ are both trace definable in $\Th(A)$.
The first case follows by Proposition~\ref{prop:ab}.
The second case follows by Fact~\ref{fact:cantor} as $A$ is not totally transcendental.
For the latter, fix a prime $p$ such that $A$ is not $p$-divisible and note that $A, pA, p^2A, \ldots$ is an infinite strictly descending chain of definable subgroups of $A$.

\medskip
We now show that $A$ is trace definable in the theory of $(\Q; +) \sqcup \Sa C$.
Let $B = A \otimes_\Z \Q$ be the divisible hull of $A$.
Then $B \equiv (\Q; +)$.
Let $\Sa U$ be the structure with domain the underlying set of $A$ and unary relations defining every coset of every $kA$.
Fact~\ref{fact:cantor} shows that $\Sa U$ is trace definable in $\Th(\Sa C)$.
It suffices to show that $A$ is trace definable in $B \sqcup \Sa U$.
By Fact~\ref{fact:abelian qe}  every $A$-definable subset of $A^n$ is a boolean combination of sets $X$ of the following forms:
\begin{enumerate}[leftmargin=*]
\item $X=\{(\beta_1, \ldots, \beta_n) \in A^n : m_1\beta_1 + \cdots + m_n \beta_n = \gamma\}$, or
\item $X=\{(\beta_1, \ldots, \beta_n) \in A^n : m_1\beta_1 + \cdots + m_n \beta_n \equiv \gamma \pmod{k}\}$
\end{enumerate} 
for some $m_1, \ldots, m_n \in \Z$, $\gamma \in A$, and $k \ge 1$.
If $X$ is as in (1) then $X$ is the intersection of a $B$-definable subset of $B^n$ with $A^n$.
Let $X$ be as in (2).
Whether $(\beta_1, \ldots, \beta_n) \in A^n$ is in $X$ depends only on the residue mod $k$ of the $\beta_i$.
As $A/k A$ is finite $X$ is definable in $\Sa U$.
Hence the inclusion $A \to B$ and the identity $A \to \Sa U$ witness trace definability of $A$ in $B \sqcup \Sa U$.
\end{proof}

Corollary~\ref{cor:Z-decomp} follows from Proposition~\ref{prop:Z-decomp}, the fact that a torsion free abelian group is superstable if and only if it is non-singular~\cite[Thm.~A.2.13]{Hodges}, and preservation of superstability under trace definability, see Fact~\ref{fact:preserve}.

\begin{corollary}\label{cor:Z-decomp}
A torsion free abelian group is trace definable in $\Th(\Z; +)$ if and only if it is superstable and is trace equivalent to $(\Z; +)$ if and only if it is superstable and not totally transcendental.
\end{corollary}

We finally prove a general result on abelian groups which will be applied to ordered abelian groups below.
Recall that a direct summand of an abelian group is a  pure subgroup.
See \cite[10.7.1, 10.7.3]{Hodges} or \cite[Cor.~3.3.38]{trans} for Fact~\ref{fact:pure}.

\begin{fact}\label{fact:pure}
If $A$ is a pure subgroup of an abelian group $B$ and the expansion of $B$ by a unary relation defining $A$ is $\aleph_1$-saturated then $A$ is a direct summand of $B$.
\end{fact}

We now give the connection between purity and trace definability.

\begin{lemma}
\label{lem:pure}
Any pure embedding between abelian groups is a trace embedding.
If $A$ is a pure subgroup of an abelian group $B$ then $\Th(B)$ trace defines both $A$ and $B/A$.
\end{lemma}

\begin{proof}
The first claim follows by Facts~\ref{fact:abelian qe} and \ref{fact:embed}.
Let $A$ be a pure subgroup of $B$ and  $(B, A)$ be the expansion of $B$ by a unary relation defining $A$.
Let $(B^*, A^*)$ be an  $\aleph_1$-saturated elementary extension of $(B, A)$.
Now $B^* \equiv B$ and $B^*/A^* \equiv B/A$, so it is enough to show that $B^*$ trace defines $B^*/A^*$.
By Fact~\ref{fact:pure} $A^*$ is a direct summand of $B^*$.
Therefore $B^*/A^*$ is also a direct summand of $B^*$, hence $B^*$ trace defines $B^*/A^*$.
\end{proof}

\subsection{Ordered abelian groups and ordered vector spaces}
We first prove the ordered analogue of Proposition~\ref{prop:lve}.

\begin{proposition}
\label{prop:lvoe}
Any ordered vector space over an ordered division ring is locally trace equivalent to $\rgoup$.
If $\E/\D$ is a finite extension of ordered division rings then the theory of ordered $\E$-vector space is trace equivalent to the theory of ordered $\D$-vector spaces.
\end{proposition}

\begin{proof}
Let $\E/\D$ be an extension of ordered division rings, let  $\Sa V$ be an ordered $\E$-vector space, let $\Sa V^*$ be the ordered $\D$-vector space reduct of $\Sa V$, and let $B$ be a basis for $\E$ as a vector space over $\D$.
It suffices to show that $(v \mapsto bv)_{b \in B}$ witnesses local trace definability of $\Sa V$ in $\Sa V^*$.
By quantifier elimination for ordered vector spaces~\cite[Cor.~7.8]{tametop} every $\Sa V$-definable subset of $V^n$ is a boolean combination of sets given by inequalities of the form $\lambda_1 v_1 + \cdots + \lambda_n v_n \ge \gamma$
for $\lambda_1, \ldots, \lambda_n \in \E$ and $\gamma \in V$.
Expressing each $\lambda_i$ as a $\D$-linear combination of elements of $B$, we see that such an inequality is equivalent to an inequality $\eta_1(b_1 v_{i_1}) + \cdots + \eta_m(b_m v_{i_m}) \ge \gamma$ for some $\eta_1, \ldots, \eta_m \in \D$, $b_1, \ldots, b_m \in B$, and $i_1, \ldots, i_m \in \nset$.
\end{proof}

We now give the ordered version of Proposition~\ref{prop:ab}.

\begin{proposition}\label{prop:re-oag}
The theory of any ordered abelian group trace defines $\rgoup$.
\end{proposition}

\begin{proof}
By Fact~\ref{fact:embed} and quantifier elimination for $\Th\rgoup$ it is enough to suppose that $\hgroup$ is an $\aleph_1$-saturated ordered abelian group and embed $(\Q;+,<)$ into $\hgroup$.
Note that if $\chi$ is any embedding $\Q \to H$ of abelian groups then either $\chi$ or $-\chi$ is an embedding of ordered abelian groups.
So we may apply the proof of Proposition~\ref{prop:ab}.
\end{proof}

Let $\hgroup$ an ordered abelian group.
Our next goal is to show that $\hgroup$ is (locally) trace equivalent to $(H;+)\sqcup \rgoup$ under certain circumstances.
We first describe an important class of ordered abelian groups.
We say that $\hgroup$ has {\bf bounded regular rank} if the following equivalent conditions hold.
\begin{enumerate}[leftmargin=*]
\item Any elementary extension  has only countably many definable convex subgroups.
\item There is an upper bound on the cardinality of the collection of definable convex subgroups in an elementary extension of $\hgroup$.
\item Any $\hgroup$-definable family of convex subgroups has only finitely many elements.
\item $\hgroup$ has finite $p$-regular rank for every prime $p$.
\end{enumerate}
The first three conditions are equivalent by a routine model-theoretic argument.
See \cite[Prop.~2.3]{Farre} for the equivalence of the first and last conditions, and for an explanation of $p$-regular rank.
It follows directly from the definition of $p$-regular rank that any non-singular ordered abelian group has bounded regular rank.
Any archimedean ordered abelian group, more generally any substructure of a lexicographic power of $\rgoup$, has bounded regular rank as it has only finitely many convex subgroups.
Any strongly dependent ordered abelian group has bounded regular rank~\cite{DGoag}.
An infinite lexicographic power of a non-divisible ordered abelian group does not have bounded regular rank.

\medskip
We let $\Cal C(H)$ be the collection of definable convex subgroups of $\hgroup$.
If $H$ is discrete then  let $1_H$ be the minimal positive element of $H$.
Given a collection $\mathcal{J}$ of convex subgroups of $H$ let $L_\mathcal{J}$ be the expansion of the language of ordered groups by:
\begin{enumerate}
\item a unary relation defining each  $J\in \mathcal{J}$,
\item a unary relation defining the coset $1_J$ for every $J\in\mathcal{J}$ such that $H/J$ is discrete,
\item and a unary relation defining $J+p^m H$ for every prime $p$, $J\in\mathcal{J}$, and $m \ge 1$.
\end{enumerate}
Let $H_\mathcal{J}$ be the natural $L_\mathcal{J}$-structure on $H$.
Note that $H_\mathcal{J}$ is interdefinable with $\hgroup$ when $\mathcal{J} = \mathcal{C}(H)$.
Fact~\ref{fact:farre} is a special case of the Gurevich-Schmitt theorem~\cite[Thm.~2.4]{Farre}.
It is usually stated in terms of a certain collection $\mathrm{RJ}(H)$ of specific definable convex subgroups, i.e. in a certain sublanguage of $L_{\mathcal{C}(H)}$.
Note that our version immediately follows.

\begin{fact}
\label{fact:farre}
If $\hgroup$ has bounded regular rank then $H_{\mathcal{C}(H)}$ admits quantifier elimination.
\end{fact}

We now show that the local trace equivalence type of an ordered abelian group of bounded regular rank is determined by the local trace equivalence type of the underlying group.

\begin{proposition}
\label{prop:bounded-reg}
\hspace{.0000000000000001cm}
\begin{enumerate}[leftmargin=*]
\item If $\mathcal{J}$ is a collection of convex subgroups of $H$ such that $H_\mathcal{J}$ has quantifier elimination then $H_\mathcal{J}$ is locally trace equivalent to $(H;+)\sqcup\rgoup$.
\item If $\hgroup$ has bounded regular rank then $\hgroup$ and $(H;+)\sqcup\rgoup$ are locally trace equivalent.
\item If $\hgroup$ has only finitely many definable convex subgroups then $\hgroup$ is trace equivalent to $(H; +) \sqcup \rgoup$. 
In particular $\hgroup$ and $(H;+)\sqcup\rgoup$ are trace equivalent when $\hgroup$ is archimedean.
\end{enumerate}
\end{proposition}

\begin{proof}
Note that (2) follows from (1) by the first claim by Fact~\ref{fact:farre}.
Furthermore (3) follows from the proof of (2) as our witness of local trace definability in (2) will only contain finitely many functions when $\hgroup$ has only finitely many definable convex subgroups.

\medskip
We prove (1).
By Propositions~\ref{prop:re-oag} and Lemma~\ref{lem:k-dis} the theory of any ordered abelian group $\hgroup$ trace defines $(H;+)\sqcup\rgoup$.
Let $(K;+,\triangleleft)$ be the divisible hull of $\hgroup$.
Then we have $(H;+)\sqcup\rgoup \equiv (H;+) \sqcup (K;+,\triangleleft)$ as $\rgoup \equiv (K;+,\triangleleft)$.
We show that $H_\mathcal{J}$ is locally trace definable in $(H;+)\sqcup(K;+,\triangleleft)$.
Let $\Sa E$ be the disjoint union of the abelian groups $H/J$ for $J\in \mathcal{J}$.
If $J$ is a convex subgroup of $H$, then $J$ is a pure subgroup of $H$, hence $(H;+)$ trace defines $H/J$ by Lemma~\ref{lem:pure}.
Therefore $(H;+)$ trace defines $\Sa E$ by Lemma~\ref{lem:k-dis}.
It is enough to show that $\Sa E\sqcup(K;+,\triangleleft)$ locally trace defines $\hgroup$.
Let $\chi \colon H \to K$ be the inclusion and let $\uptau_J \colon H \to H/J$ be the quotient map for every $J \in \mathcal{J}$.
We show that $\chi$ and the $\uptau_J$ witness local trace definability of $\hgroup$ in $\Sa E\sqcup(K;+,\triangleleft)$.
By quantifier elimination every $H_\mathcal{J}$-definable subset of $H^n$ is a boolean combination of sets $X$ where $X$ is the set of $(\alpha_1,\ldots,\alpha_n)$ satisfying a condition of one of the following forms for integers $m_1, \ldots, m_n$, $\rho \in H$, and $J \in\mathcal{J}$.
\begin{enumerate}[leftmargin=*]
\item $m_1\alpha_1+\cdots+m_n\alpha_n+\rho \triangleright 0$.
\item $ m_1\alpha_1+\cdots+m_n\alpha_n+\rho\in J$.
\item $  m_1\alpha_1+\cdots+m_n\alpha_n+\rho\in J+p^kH$ for a prime $p$ and $k \ge 1$.
\end{enumerate}

First suppose that $X$ is as in (1).
Let $Y$ be the set of $(\beta_1,\ldots,\beta_n)\in K^n$ such that we have $0 \triangleleft m_1\beta_1+\cdots+m_n\beta_n+\uptau(\rho)$.
Then $Y$ is $\Sa E\sqcup(K;+,\triangleleft)$-definable and we have $\alpha \in X$ if and only if $\uptau(\alpha)\in Y$ for all $\alpha \in H^n$.
We suppose that $X$ is as in (3), the case when $X$ is as in (2) follows in the same way.
Now let $Y$ be the set of $(\beta_1,\ldots,\beta_n)\in(H/J)^n$ such that $m_1\beta_1+\cdots+m_n\beta_n+\uptau(\rho)$ is in $p^m(H/J)$.
Then $Y$ is $\Sa E\sqcup(K;+,\triangleleft)$-definable and we have $\alpha\in X$ if and only if $\uptau(\alpha)\in Y$ for all $\alpha \in H^n$.
\end{proof}

\begin{corollary}\label{cor:zaran oag}
Any ordered abelian group with bounded regular rank has near linear Zarankiewicz bounds.
\end{corollary}



Corollary~\ref{cor:zaran oag} follows by Proposition~\ref{prop:bounded-reg} and Lemma~\ref{lem:vector space}.
We now prove a result about ordered vector spaces which will be useful when dealing with o-minimal expansions of ordered groups.
Let $\D$ be an ordered division ring and $\Sa V$ be an ordered $\D$-vector space.
Given a division subring $\D^*$ of $\D$ we let $\Sa V^*$ be the ordered $\D^*$-vector space reduct of $\Sa V$.

\begin{lemma}
\label{lem:ovs}
Let $\D$ and $\Sa V$ be as above.
Suppose that $\Sa M$ is a reduct of $\Sa V$ expanding $(V;+,<)$.
Then there is a division subring $\D^*$ of $\D$ such that $\Sa M$ is trace equivalent to $\Sa V^*$.
\end{lemma}


\begin{proof}
Let $\Sa M'$ be the expansion of $(V;+,<)$ by all  $f\colon [0,\gamma)\to V$ such that:
\begin{enumerate}
\item $\gamma\in V\cup\{\infty\}$,
\item $f$ is definable in $\Sa M$,
\item there is $\lambda\in\D$ such that $f(v)=\lambda v$ for all $0\le v<\gamma$,
\end{enumerate}
Then $\Sa M'$ is a reduct of $\Sa M$.
We sketch a proof that $\Sa M$ and $\Sa M'$ are interdefinable.
Following the semilinear cell decomposition \cite[1.7.8]{lou-book} one can show that any $\Sa M$-definable subset of $V^n$ is a finite union of $\Sa M$-definable semilinear cells.
An easy induction on $n$ shows that any $\Sa M$-definable semilinear cell $X \subseteq V^n$ is definable in $\Sa M'$.
So we suppose $\Sa M=\Sa M'$.

\medskip
Let $\D^*$ be the set of $\lambda\in\D$ such that the function $[0,\gamma)\to V$ given by $v\mapsto\lambda v$ is $\Sa M$-definable for some positive $\gamma\in V$.
Note that $\D^*$ is a division subring of $\D$.
We show that $\Sa M$ and $\Sa V^*$ are trace equivalent.
First note that $\Sa M$ is a reduct of $\Sa V^*$.
We show that $\Th(\Sa M)$ trace defines $\Sa V^*$.
Suppose that $\Sa V\prec\Sa W$ is $|\D|^+$-saturated and let $\Sa N$ be the reduct of $\Sa W$ such that $\Sa M\prec\Sa N$.
Then $\Sa N$ is a $|\D|^+$-saturated elementary extension of $\Sa M$.
It is enough to show that $\Sa N$ trace defines some ordered $\D^*$-vector space as the theory of ordered $\D^*$-vector spaces is complete.
Let $I$ be the set of $w\in W$ such that $|w|<|v|$ for all non-zero $v \in V$.
Note that $I$ is a convex subgroup of $(W;+,<)$ and $I$ is non-trivial by saturation.
If $\lambda\in\D^*\setminus\{0\}$, $w \in I$, and $v \in V\setminus\{0\}$, then $|w|<|\lambda^{-1}v|$ as $\lambda^{-1}v\in V\setminus\{0\}$, so multiplying through by $|\lambda|$ yields $|\lambda w|<|v|$.
Hence $I$ is closed under multiplication by any $\lambda\in\D^*$, so $(I;+,<)$ has a natural $\D^*$-vector space expansion $\Sa I$.
Now for every $\lambda\in\D^*$ there is an $\Sa N$-definable $f\colon W\to W$ such that $f(w)=\lambda w$ for all $w \in I$.
Hence the inclusion $\Sa I \to \Sa N$ is a trace embedding by quantifier elimination for ordered vector spaces.
\end{proof}

Finally, we prove Lemma~\ref{lem:Z}.
This is generalized in Proposition~\ref{prop:big order} below.

\begin{lemma}\label{lem:Z}
$\rgoup$ is trace equivalent to $\zgoup$.
\end{lemma}

\begin{proof}
By Proposition~\ref{prop:re-oag} it suffices to show that $\zgoup$ is trace definable in $\Th\rgoup$.
Proposition~\ref{prop:Z-decomp} shows that $(\Z; +)$ is trace definable in $\Th\rgoup$, so it is enough to show that $\zgoup$ is trace definable in $(\Z; +) \sqcup \rgoup$.
An easy application of the quantifier elimination for Presburger arthimetic shows that the map $\Z \to \Z \times \R$ given by $m \mapsto (m, m)$ is a trace definition $\zgoup \rightsquigarrow (\Z; +) \sqcup \rgoup$.
\end{proof}

\subsection{Cyclically ordered abelian groups and the trace equivalence class of $\rgoup$}
We first recall some background.
A \textbf{cyclic order} on a set $J$ is a ternary relation $\cyc$ such that for all $a,b,c \in J$ we have:
\begin{enumerate}[leftmargin=*]
\item If $\cyc(a,b,c)$ then $a,b,c$ are distinct.
\item If $\cyc(a,b,c)$ and $\cyc(a,c,d)$ then $\cyc(a,b,d)$.
\item If $a,b,c$ are distinct, then either $\cyc(a,b,c)$ or $\cyc(c,b,a)$.
\item  $\cyc(a,b,c)$, implies $\cyc(b,c,a)$ and $\cyc(a,b,c)$, implies $\cyc(c,b,a)$. 
\end{enumerate}

If $\triangleleft$ is a linear order on $J$ then we define a cyclic order $C_\triangleleft$ on $J$ by declaring 
\[
C_\triangleleft(a,b,c) \quad\Longleftrightarrow\quad (a\triangleleft b\triangleleft c)\vee(b\triangleleft c\triangleleft a)\vee(c\triangleleft a\triangleleft b)\quad\text{for all 
 } a,b,c\in J.
 \]

\medskip
A \textbf{cyclic group order} on an abelian group $(J;+)$ is a cyclic order preserved under translation.
In this case, we call  $(J; +, \cyc)$ a \textbf{cyclically ordered abelian group}.
The obvious example is the counterclockwise cyclic order on $\R/\Z$.
A cyclically ordered abelian group is archimedean if it embeds into this.


\medskip
Suppose that $\hgroup$ is  an ordered abelian group and $u \in H$ is positive.
Let $\I=[0,u)$ and let $\oplus_u$ be the binary operation on $\I$ given by $\beta\oplus_u\beta^*=\beta+\beta^*$ when $\beta+\beta^*<u$ and $\beta\oplus_u\beta^*=\beta+\beta^*-u$ otherwise.
Then $(\I;\oplus_u,C_\triangleleft)$ is a cyclically ordered abelian group.
We describe a second   view of this construction.


\medskip
Let $\jgroup$ be a cyclically ordered abelian group.
Then $(H;+,\triangleleft,u,\uppi)$ is the \textbf{universal cover} of $\jgroup$ if  $(H; +,\triangleleft)$ is an ordered abelian group, $u$ is a positive element of $H$ such that $u\Z$ is cofinal in $H$, $\uppi\colon H \to J$ is a surjective group morphism with kernel $u\Z$, and for any $a,b,c\in\I$ we have $\cyc(\uppi(a), \uppi(b), \uppi(c))$ if and only if $C_\triangleleft(a,b,c)$.
Note that the restriction of $\uppi$ to $\I$ gives an isomorphism $(\I;\oplus_u,C_\triangleleft)\to\jgroup$.
Any cyclically ordered abelian group has a universal cover which is unique up to unique isomorphism~\cite{rieger}.
We will often drop $\uppi$ when it is clear from context.
Abusing terminology, we also refer to $\hgroup$ as the universal cover of $\jgroup$.
 
\begin{lemma}\label{lem:ce-oag}
The theory of any infinite cyclically ordered abelian group trace defines $\rgoup$.
\end{lemma}

\begin{proof}
Suppose that $\jgroup$ is an infinite cyclically ordered abelian group with universal cover $(H; +, \triangleleft, u)$.
Let $I = (-u, u)$ and let $R_+$ be the ternary relation on $I$ given by declaring $R_+(a, b, c)$ when $a + b = c$.
It is shown in \cite[\S~6.1]{SW-dp} that $(I; R_+, \triangleleft)$ is interpretable in $\jgroup$.
Hence we can follow the proof of Proposition~\ref{prop:re-oag}.
\end{proof}

Now suppose that $\Sa J$ is an expansion of $\jgroup$.
We say that the universal cover of $\Sa J$ is the expansion of $\hgroup$ by all sets of the form $\uppi^{-1}(X) \cap \I^n$ for $\Sa J$-definable $X \subseteq J^n$.

\begin{proposition}\label{prop:u cover}
Let $\jgroup$ be an infinite cyclically ordered abelian group with universal cover $(H; +, \triangleleft, u)$, $\Sa J$ be an expansion of $\jgroup$, and $\Sa H$ be the universal cover of $\Sa J$.
Then $\Sa J$ is trace equivalent to both $\Sa H$ and $(\Sa H, u\Z)$.
\end{proposition}

Hence any infinite cyclically ordered abelian group is trace equivalent to its universal cover.

\begin{proof}
It is clear from the definitions that $\Sa H$ interprets $\Sa J$.
Hence it is enough to show that $(\Sa H, u\Z)$ is trace definable in $\Th(\Sa J)$.
By Lemmas~\ref{lem:ce-oag} and \ref{lem:Z} $\Th(\Sa J)$ trace defines $\zgoup$.
It therefore suffices to show that $(\Sa H, u\Z)$ is interpretable in $\Sa J \sqcup \zgoup$.
This follows immediately from the usual construction of the universal cover, which we now recall.

\medskip
Let $\prec$ be the linear order on $J$ given by declaring $a \prec b$ when either  $a= 0 \ne b$ or $C(0, a, b)$.
Let $\lex$ be the resulting lexicographic order on $\Z \times J$.
Let $\oplus$ be the binary operation on $\Z \times J$ given by letting $(m, a) \oplus (m^*, a^*)$ be $(m + m^*, a + a^*)$ when either $a = 0$, $a^* = 0$, or $C(0, a, a + a^*)$, and otherwise $(m + m^* + 1, a + a^*)$.
Finally, let $\uppi$ be the projection $\Z \times J \to J$.
Then $(\Z \times J; \oplus, \lex, (1, 0), \uppi)$ is the universal cover of $\jgroup$.
\end{proof}

We now show that a number of structures are trace equivalent to $\rgoup$.

\begin{proposition}\label{prop:big order}
The following structures are trace equivalent to $\rgoup$.
\begin{enumerate}[leftmargin=*]
\item $(\R; +, <, \Z, \Q)$.
\item Any non-singular  ordered abelian group.
\item Any infinite non-singular cyclically ordered abelian group.
\item $(\R;+,<,t\mapsto\lambda t)$ for a fixed irrational algebraic real number $\lambda$.
\end{enumerate}
\end{proposition}

In particular any finite rank ordered abelian group is trace equivalent to $\rgoup$.
By Facts~\ref{fact:simon} and \ref{fact:preserve}  $\Th\rgoup$ cannot trace define a proper expansion of $\zgoup$.

\begin{proof}
First note that (4) follows from Proposition~\ref{prop:lvoe}.
We prove (1).
Proposition~\ref{prop:u cover}  shows that $(\R; +, <, \Z, \Q)$ is trace equivalent to $(\R/\Z; +, C, \Q/\Z)$, where $C$ is the counterclockwise cyclic order on $\R/\Z$.
Now $(\R/\Z; +, C, \Q/\Z)$ is interpretable in $(\R; +, <, \Q)$, so it suffices to show that $\rgoup$ trace defines $(\R, +, <, \Q)$.
By \cite{DMS1} $(\R;+,<,\Q)$ admits quantifier elimination after we add a unary function for scalar multiplication $\R \to \R$ by each rational.
Hence every $(\R;+,<,\Q)$-definable subset of $\R^n$ is a boolean combination of sets of one of the following forms for some $q_1, \ldots, q_n \in \Q$ and $\rho\in\R$.
\begin{enumerate}[label=(\alph*)]
\item $\{ (\alpha_1, \ldots, \alpha_n)\in \R^n: q_1\alpha_1 + \cdots +  q_n\alpha_n\ge \rho \}$.
\item  $\{(\alpha_1, \ldots, \alpha_n)\in \R^n: q_1\alpha_1 + \cdots +  q_n\alpha_n \in \rho + \Q\}$.
\end{enumerate}
Let $\mathrm{id}_\R$ be the identity $\R\to\R$.
Note that  $(\R/\Q;+)$ is a continuum size $\Q$-vector space, hence there is an isomorphism $(\R/\Q;+) \to (\R;+)$.
Let $\chi\colon \R\to \R$ be the composition of the quotient map $\R\to\R/\Q$ with such an isomorphism.
Then $\mathrm{id}_\R,\chi$ witnesses trace definability of $(\R;+,<,\Q)$ in $\rgoup$.
Here $\mathrm{id}_\R$ handles (a) and  in case (b) we have $(\alpha_1,\ldots,\alpha_n)\in X$ if and only if $q_1\chi(\alpha_1) + \cdots + q_n\chi(\alpha_n) = \chi(\rho)$ for any $(\alpha_1,\ldots,\alpha_n)\in\R^n$.

\medskip
The universal cover of a non-singular cyclically ordered abelian group is non-singular~\cite[Lemma~6.1]{SW-dp}.
Hence (3) follows from (2) and Proposition~\ref{prop:u cover}.
We finally prove (2).
Fix a non-singular ordered abelian group $\hgroup$.
By Proposition~\ref{prop:re-oag} it suffices to show that $\hgroup$ is trace definable in $\Th\rgoup$. 
Proposition~\ref{prop:Z-decomp}  shows that $\Th\rgoup$ trace defines $(H; +)$.
Hence it suffices to show that $\hgroup$ is trace definable in the theory of $(H;+) \sqcup \rgoup$.
We adapt the proof of Proposition~\ref{prop:bounded-reg}.
Let $(K; +, \triangleleft)$ be as in that proof and let $\Sa K$ be the Shelah completion of $(K; +, \triangleleft)$.
Then $(K; +, \triangleleft)$ is trace equivalent to $\Sa K$ by Proposition~\ref{prop:she-0}, so it is enough to show that $\hgroup$ is trace definable in $(H ; +) \sqcup \Sa K$.
As in the proof of Proposition~\ref{prop:bounded-reg} any definable subset of $H^n$ is a boolean combination of sets $X$ where $X$ is the set of $(\alpha_1,\ldots,\alpha_n) \in H^n$ satisfying a condition of one of the following forms for integers $m_1, \ldots, m_n$, $\rho \in H$, and a definable convex subgroup $J \subseteq H$.
\begin{enumerate}[label=(\alph*), leftmargin=*]
\item $m_1\alpha_1+\cdots+m_n\alpha_n+\rho \triangleright 0$.
\item $ m_1\alpha_1+\cdots+m_n\alpha_n+\rho\in J$.
\item $  m_1\alpha_1+\cdots+m_n\alpha_n+\rho\in J+p^kH$ for a prime $p$ and $k\in\N$.
\end{enumerate}
We handle (a) as in the proof of \ref{prop:bounded-reg}.
Suppose $X$ is as in (b).
Let $J^*$ be the convex hull of $J$ in $K$, so $J^*$ is $\Sa K$-definable by convexity.
Hence $X$ is of the form $Y \cap H^n$ for $\Sa K$--definable $Y \subseteq K^n$.
Finally, suppose that $X$ is as in (c).
Then $p^k H$ has finite index in $H$, so $J + p^k H$ is a finite union of cosets of $p^k H$.
Hence $X$ is definable in $(H; +)$.
\end{proof}


\begin{fact}\label{fact:lex}
Suppose that $\hgroup$ is an ordered abelian group, $K$ is a convex subgroup of $H$, and $(H; +, \triangleleft, K)$ is $\aleph_1$-saturated.
Then $\hgroup$ is isomorphic to the lexicographic product $H/K \times K$, where $K$ and $H/K$ are both given the induced ordered group structure.
\end{fact}


\begin{proof}
First note that $K$ is a pure subgroup of $H$ by convexity.
Hence by Fact~\ref{fact:pure} there is a group isomorphism $H \to H/K \oplus K$.
Note that $0 \to K \to H \to H/K \to 0$ gives a short exact sequence of ordered abelian groups, i.e. a short exact sequence of abelian groups where the morphisms are homomorphisms of ordered groups.
It follows that $H \to H/K \oplus K$ gives an ordered group isomorphism between $\hgroup$ and the lexicographic order on $H/K \oplus K$.
\end{proof}

\begin{proposition}\label{prop:nscover}
Let $\hgroup$ be a non-singular ordered abelian group, $u$ be a positive element of $H$, and $\Sa H$ be an expansion of $\hgroup$ by a collection $\Cal X$ of subsets of the $\I^n$.
Then $\Sa H$ is trace equivalent to the structure $\Sa J$ induced on $\I$ by $\Sa H$.
\end{proposition}

\begin{proof}
Note that $\Sa H$ is interdefinable with the expansion of $\hgroup$ by the collection of all $\Sa H$-definable subsets of all $\I^n$.
Hence we may suppose that $\Cal X$ is the collection of all $\Sa J$-definable sets.
Let $K$ be the convex hull of $u\Z$ in $H$ and let $\Sa K$ be  $(K ; +, \triangleleft, \Cal X)$.
Consider $\Sa J$ as an expansion of the cyclically ordered group $(\I; \oplus_u, C_\triangleleft)$.
Then $\Sa K$ is the universal cover of $\Sa J$, hence $\Sa K$ is trace equivalent to $\Sa J$ by Proposition~\ref{prop:u cover}.
We show that $\Th(\Sa K)$ trace defines $\Sa H$.
After possibly passing to an elementary extension we suppose that the expansion of $\Sa H$ by $K$ is $\aleph_1$-saturated.
Let $G = H/K$.
As $K$ is a convex subgroup we consider $G$ to be an ordered abelian group.
Furthermore $G$ is non-singular as it is a quotient of a non-singular abelian group.
Hence $G$ is trace definable in $\Th(K; +, \triangleleft)$ by Proposition~\ref{prop:big order}.
By Fact~\ref{fact:lex} $\hgroup$ is isomorphic to the lexicographic product of $G$ with $(K; +, \triangleleft)$, and this lexicographic product is definable in $(K; +, \triangleleft) \sqcup G$.
It follows that $\Sa H$ is isomorphic to a structure definable in $\Sa K \sqcup G$.
\end{proof}

\section{Dp-minimal and o-minimal expansions of ordered groups}
\label{section:dp-min oag}
Laskowski and Steinhorn~\cite{LasStein} showed that any o-minimal expansion of an archimedean ordered abelian group is elementarily equivalent to an o-minimal expansion of $\rgoup$.
Their result obviously does not extend beyond o-minimal structures.
We first prove Theorem~\ref{thm:dp} and then use it to give a dichotomy between field structure and linearity in Section~\ref{section:dicho}.

\begin{theorem}\label{thm:dp}
Suppose that $(R;+,<)$ is an archimedean ordered abelian group and $\Sa R$ is a dp-minimal expansion of $(R;+,<)$.
Then $\Sa R$ is trace equivalent to an o-minimal expansion $\Sq R$ of $\rgoup$ such that $\Sq R$ is interpretable in the Shelah completion of any $\aleph_1$-saturated elementary extension of $\Sa R$.
\end{theorem}

Fact~\ref{fact:simon} summarizes what is known about dp-minimal expansions of archimedean ordered abelian groups.

\begin{fact}
\label{fact:simon}
Suppose that $(R;+,<)$ is an archimedean ordered abelian group and $\Sa R$ is an expansion of $(R;+,<)$.
\begin{enumerate}[leftmargin=*]
\item  $\Sa R$ is dp-minimal if and only if the following holds for any elementary extension $\Sa S$ of $\Sa R$: every definable subset of $S$ is a finite disjoint union of sets of the form $I \cap [nS + \beta]$ for convex $I\subseteq S$, $n\in\N$, and $\beta\in S$.
\item If $(R;+,<)$ is divisible then $\Sa R$ is dp-minimal if and only if $\Sa R$ is weakly o-minimal.
\item If $(R;+,<)=(\R;+,<)$ then $\Sa R$ is dp-minimal if and only if $\Sa R$ is o-minimal.
\item There are no proper strongly dependent expansions of $(\Z;+,<)$.
\end{enumerate}
\end{fact}

Here (1) is proven in \cite{SW-dp}, (2) follows from (1), (3) follows from (2), and (4) is due to Dolich-Goodrick~\cite{DG}.
(Recall that dp-minimality implies strong dependence.)
Note (3) was first proven in \cite{Simon-dp}.
We let $\dprk(\Sa M)$ be the dp-rank of a structure $\Sa M$ and $\dprk_{\Sa M}(X)$ be the dp-rank of an $\Sa M$-definable set $X$, we drop the subscript when it is clear from context.
See \cite[Chap.~4]{Simon-Book} for background on dp-rank.
Fact~\ref{fact:dp} contains some basic facts about dp-rank.

\begin{fact}\label{fact:dp}
Let $\Sa M$ be a $\nip$ structures and $X,Y$ be $\Sa M$-definable sets.
Then we have the following.
\begin{enumerate}[leftmargin=*]
\item $\dprk(X \times Y) \le \dprk(X) + \dprk(Y)$.
\item If there is a definable surjection $X \to Y$ then $\dprk(Y) \le \dprk(X)$.
\item If there is a trace embedding $\Sa M \to \Sa N$ then $\dprk(\Sa M) \le \dprk(\Sa N)$.
\item $\dprk_{\Sa M}(X) = \dprk_{\Sh M}(X)$ and in particular $\dprk(\Sh M) = \dprk(\Sa M)$.
\item If $\Sa M$ is dp-minimal then $\dprk(M^n) = n$ for all $n$.
\end{enumerate}
\end{fact}

Here (3) is immediate from the definition of the dp-rank, (3) is proven in \cite[Prop.~4.2]{trace1}, and (4) follows by (3) and Lemma~\ref{lem:she-last}.
See \cite[Prop.~4.20]{Simon-Book} for (1).
Finally, it follows from the definition of dp-rank that $\dprk_{\Sa M}(M^n) \ge n$ for any structure $\Sa M$ and $n \ge 1$.
Hence (5) follows from (1).

\medskip
We first treat the discrete case of Theorem~\ref{thm:dp}.
By Fact~\ref{fact:simon}(4) we may suppose that $\Sa R$ is $(\Z;+,<)$.
In this case we take $\Sq R$ to be $\rgoup$.
By Lemma~\ref{lem:Z}  it is enough to show that $\rgoup$ is interpretable in the Shelah completion of any $\aleph_1$-saturated elementary extension of $(\Z;+,<)$.
If $(Z;+,<)$ is such an extension then we fix positive $\gamma \in Z \setminus \Z$, let $V$ be the convex hull of $\gamma\Z$, and let $\mfrak$ be the set of $a \in Z$ such that $n|a| < \gamma$ for all $n$.
Then $V$ and $\mfrak$ are both definable in $(Z ; + , <)^{\mathrm{Sh}}$ and $V/\mfrak$ is isomorphic to $\rgoup$.

\medskip
{\bf We suppose throughout the remainder of this section that $(R;+,<)$ is a dense archimedean ordered abelian group and $\Sa R$ is a dp-minimal expansion of $(R;+,<)$.}
We first define the  o-minimal completion $\Sq R$ of $\Sa R$.
Now $(R;+,<)$ admits a unique up to rescaling embedding into $(\R;+,<)$, so we suppose that $(R;+,<)$ is a substructure of $\rgoup$.
Let $\Sa R\prec\Sa N$ be $\aleph_1$-saturated.
Declare
\begin{align*}
V &= \{ a \in N : |a| < n, \text{ for some } n\in\N\}\\
\mfrak &= \{ a \in N : |a| < 1/n, \text{ for all } n\in\N, n \ge 1 \}.
\end{align*}
Then $V$ and $\mfrak$ are convex subgroups of $N$ and by saturation we may identify $V/\mfrak$ with $\R$ and the quotient map $V \to \R$ with the usual standard part map $\st\colon V \to \R$.
Note that $V$ and $\mfrak$ are externally definable, so we consider $\R$ to be an $\Sh N$-definable set of imaginaries.
We let $\Sq R$ be the structure induced on $\R$ by $\Sh N$.
Fact~\ref{fact:Sq} is proven in \cite{big-nip} more generally for strongly dependent expansions of archimedean ordered abelian groups.

\begin{fact}
\label{fact:Sq}
The structure induced on $R$ by $\Sq R$ is a reduct of $\Sh R$.
\end{fact}

Lemma~\ref{lem:sq-o-min} follows by Facts~\ref{fact:dp}(2) and \ref{fact:simon}(4).

\begin{lemma}
\label{lem:sq-o-min}
$\Sq R$ is o-minimal.
\end{lemma}


We let $\dim X$ be the usual o-minimal dimension of an $\Sq R$-definable subset of $\R^n$ and $\cl(X)$ be the closure in $\R^n$ of $X\subseteq\R^n$.
We apply the well-known and easy to prove fact that o-minimal dimension and dp-rank agree over o-minimal structures.

\begin{lemma}
\label{lem:dim-inequality}
Suppose that $X$ is an $\Sa R$-definable subset of $R^n$.
Then $\dim \cl(X) = \dprk_{\Sa R} X$.
\end{lemma}

\begin{proof}
We first show that $\dim \cl(X) \le \dprk_{\Sa R} X$.
Let $X^*$ be the subset of $\Sa N$ defined by the same formula as $X$.
A routine saturation argument shows that $\cl(X)=\st(X^* \cap V^n)$.
By Fact~\ref{fact:dp} we have
\[
\dprk_{\Sh N} \cl(X) \le \dprk_{\Sh N} X^* \cap V^n \le \dprk_{\Sh N} X^*=\dprk_{\Sa N} X^* = \dprk_{\Sa R} X.
\]
Finally, $\Sq R$ is the structure induced on $\R$ by $\Sh N$ hence 
\[\dprk_{\Sh N} \cl(X)=\dprk_{\Sq R} \cl(X)=\dim \cl(X).\]

\medskip
We now show that $\dim\cl(X)\ge\dprk_{\Sa R}X$.
Passing to $\Sh R$ does not change the dp-rank of $\Sa R$-definable sets and by definition $(\Sh R)^\square$ is interdefinable with $\Sq R$.
So after possibly replacing $\Sa R$ with $\Sh R$ we suppose that $\Sa R$ is Shelah complete.
By Fact~\ref{fact:dp}(5) we have $\dprk_{\Sa R}R^n=n$ for all $n$.
Let $Y_1,\ldots,Y_k$ be a partition of $\cl(X)$ into $\Sq R$-definable cells.
By Fact~\ref{fact:Sq} and Shelah completeness each $X\cap Y_i$ is $\Sa R$-definable.
We have
\[
\dim \cl(X) = \max\{\dim Y_1,\ldots,\dim Y_k\} \quad \text{and}\quad \dprk_{\Sa R} X = \max\{\dprk_{\Sa R} X\cap Y_1,\ldots,\dprk_{\Sa R} X \cap Y_k \}.
\]
Hence it is enough to show that $\dprk_{\Sa R}X\cap Y_i \le \dim Y_i$ for each $i = 1,\ldots,k$.
Fix $i$ and let $d = \dim Y_i$.
As $Y_i$ is a cell there is a coordinate projection $\uppi\colon \R^n \to \R^d$ such that the restriction of $\uppi$ to $Y_i$ is injective.
By Fact~\ref{fact:dp} 
\[\dprk_{\Sa R} X\cap Y_i = \dprk_{\Sa R}\uppi(X \cap Y_i) \le \dprk_{\Sa R}R^d=d.\]
\end{proof}

We will now need to work with cuts.
For our purposes a \textbf{cut} in $R$ is a nonempty downwards closed bounded above set $C \subseteq R$ such that either $C$ does not have a supremum in $R$ or $C$ does not contain its supremum.
So if $r \in R$ then $(-\infty,r)$ is a cut but $(-\infty,r]$ is not.
We identify each cut in $R$ with its supremum in $\R$.
We let $\overline{R}$ be the set of  definable cuts in $R$.
Note that if $\Sa R$ is Shelah complete then $\overline{R}=\R$.
We equip $\overline{R}$ with the inclusion ordering on cuts.
We identify each $r \in R$ with $(-\infty,r)$ and hence consider $R$ to be a subset of $\overline{R}$.

\medskip
It would be convenient if there was a single definable family of cuts which contained every element of $\overline{R}$, as we could then think of $\overline{R}$ as an $\Sa R$-definable set of imaginaries.
However this is generally not the case.
We can naturally view $\overline{R}$ as an ind-definable set, e.g. an object in the ind-category of the category of definable sets of imaginaries and definable maps.
See \cite{kamensky} for a description.
The following definitions follow this formalism, but they are also natural enough to be taken on their own.
We say that a subset $X$ of $\overline{R}^n$ is definable if $X=\{ (C^1_\beta,\ldots,C^n_\beta) : \beta \in R^m\}$ for a definable family $( (C^1_\beta,\ldots,C^n_\beta): \beta \in R^m)$ of $n$-tuples of definable cuts.
Let $Y\subseteq R^m$ be definable.
Then a function $f\colon Y \to \overline{R}^n$ is definable if $( f(\beta) : \beta \in Y)$ is a definable family of $n$-tuples of cuts.
Finally we say that a function $f \colon \overline{R}^n \to Y$ is definable if the function $f\circ g \colon X \to Y$ is definable for any definable $X\subseteq R^m, g \colon X \to \overline{R}^n$.

\begin{lemma}
\label{lem:ind}
If $X$ is an $\Sa R$-definable subset of $\overline{R}^n$ then $\cl(X)$ is an $\Sq R$-definable subset of $\R^n$.
If $Y\subseteq R^m$ and $f\colon Y \to \overline{R}^n$ are definable then $\dim\cl(f(Y)) \le \dprk_{\Sa R} Y$.
\end{lemma}

We let $\overline{V}$ be the convex hull of $V$ in $\overline{N}$.
Note that $\st\colon V\to \R$ extends uniquely to a monotone map $\overline{V}\to\R$ which we also denote by $\st$.

\begin{proof}
Let $X^*$ be the subset of $\overline{N}^n$ defined by the same formula as $X$.
Again, a standard saturation argument shows that $\cl(X)=\st(X^* \cap \overline{V}^n)$, so $\cl(X)$ is $\Sh N$-definable.
Let $f^*\colon Y^*\to \overline{N}^n$ be the function defined by the same formula as $f$.
The map $Y^*\to \cl(f(Y))$ given by $\beta\mapsto \st(f(\beta))$ is an $\Sh N$-definable surjection.
The proof of Lemma~\ref{lem:dim-inequality} shows that $\dim\cl(f(Y))\le\dprk_{\Sa R}Y$.
\end{proof}

\begin{lemma}
\label{lem:combination}
Suppose that $\Sa R$ is Shelah complete.
Then $X\subseteq R^n$ is $\Sa R$-definable if and only if $X$ is a boolean combination of $(R;+)$-definable sets and sets of the form $Y \cap R^n$ for $\Sq R$-definable $Y \subseteq\R^n$.
\end{lemma}

We let $\Gamma(f)$ be the graph of a function $f$.

\begin{proof}
Let $\Cal B$ be the collection of boolean combinations of $(R;+)$-definable sets and sets of the form $Y \cap R^n$ for $\Sq R$-definable $Y\subseteq \R^n$.
We first show that $\Cal B$ is closed under products.
If $X \subseteq R^n$ and $Y \subseteq R^m$ then $X \times Y = [X \times R^m] \cap [R^n \times Y]$.
Hence it is enough to show that $X \times R^m$ is in $\Cal B$ when $X \subseteq R^n$ is $\Cal B$.
Now $X \mapsto X \times R^m$ gives an embedding from the boolean algebra of subsets of $R^n$ to the boolean algebra of subsets of $R^{n + m}$.
Hence it is enough to note that $X \times R^m$ is $(R; +)$-definable when $X$ is $(R; +)$-definable and that $(Y \cap R^n) \times R^m = (Y \times \R^n) \cap R^{n + m}$ when $Y$ is $\Sq R$-definable.

\medskip
By Fact~\ref{fact:Sq} and Shelah completeness of $\Sa R$ every set in $\Cal B$ is $\Sa R$-definable.
We suppose that $X\subseteq R^n$ is $\Sa R$-definable and show that $X$ is in $\Cal B$.
We apply induction on $n$.
The case when $n=1$ follows by Fact~\ref{fact:simon} and the fact that every convex subset of $R$ is of the form $I\cap R$ for an interval $I\subseteq\R$.
Suppose $n\ge 2$.
Let $\uppi\colon R^n \to R^{n-1}$ be the projection away from the last coordinate and let $Y=\uppi(X)$.
We also apply induction on $\dim \cl(Y)$.
If $\dim \cl(Y)=0$ then $\cl(Y)$ is finite, hence $Y$ is finite.
The case $n = 1$ shows that each $X_\beta$ is in $\Cal B$ for all $\beta\in Y$.
Hence $X = \bigcup_{\beta \in Y} [\{\beta\} \times X_\beta]$ is in $\Cal B$.
So we suppose that $\dim \cl(Y) \ge 1$.

\medskip
Fact~\ref{fact:simon} and a standard compactness argument together show that there are $\Sa R$-definable $X_1,\ldots,X_k$ and  $\Sa R$-definable families $(I^1_\beta : \beta \in Y),\ldots,(I^k_\beta : \beta \in Y)$ of nonempty convex open subsets of $R$ such that  $X=X_1\cup\cdots\cup X_k$ and for each $i = 1,\ldots,k$ we either have:
\begin{enumerate}[leftmargin=*]
\item there is $\ell$ such that $|(X_i)_\beta|\le\ell$ for all $\beta\in R^{n-1}$, or
\item there is $m>0$, $\gamma\in R$ such that either $(X_i)_\beta=\emptyset$ or $(X_i)_\beta=I^i_\beta\cap[mR+\gamma]$.
\end{enumerate}
It is enough to show that each $X_i$ is in $\Cal B$.
Hence we may suppose that there is a definable family $(I_\beta : \beta \in Y)$ of convex open subsets of $R$ such that either $X$ satisfies (1) or $X$ satisfies (2) with respect to $(I_\beta : \beta \in Y)$.
We first suppose that $X$ satisfies (1).
Algebraic closure and definable closure agree in $\Sa R$ as $\Sa R$ expands a linear order.
Hence we may suppose that $|X_\beta|\le 1$ for all $\beta\in R^{n-1}$.
Then the projection $X \to Y$ is bijective, hence $\dprk_{\Sa R}X=\dprk_{\Sa R} Y$. 
By Lemma~\ref{lem:dim-inequality} we have
\[
\dim \cl(X) = \dprk_{\Sa R} X = \dprk_{\Sa R} Y = \dim\cl(Y).
\]
Applying the o-minimal fiber lemma we partition $\cl(Y)$ into disjoint $\Sq R$-definable sets $Z_0,Z_1$ such that $\dim Z_0 < \dim \cl(Y)$ and $\cl(X)_\beta$ is finite for all $\beta \in Z_1$.
By induction $X \cap \uppi^{-1}(Z_0)$ is in $\Cal B$, so it is enough to show that $X \cap \uppi^{-1}(Z_1)$ is in $\Cal B$.
By agreement of algebraic and definable choice in $\Sq R$ there is a partition $W_1,\ldots,W_k$ of $Z_1$ into pairwise disjoint $\Sq R$-definable sets and $\Sq R$-definable functions $f_1 \colon W_1 \to R,\ldots,f_k \colon W_k \to R$ such that $\cl(X) \cap \uppi^{-1}(Z_1)$ agrees with $\Gamma(f_1)\cup\cdots\cup \Gamma(f_k)$.
It is enough to fix $i = 1,\ldots,k$ and show that $X^*:=X \cap \Gamma(f_i)$ is in $\Cal B$.
By induction $\uppi(X^*)$ is in $\Cal B$.
Note that
\begin{align*}
X^* &= \{ (\beta,\alpha)\in R^{n-1}\times R :\beta\in \uppi(X^*) \text{  and  } \alpha=f_i(\beta)\} \\
&= [\uppi(X^*)\times R] \cap \Gamma(f_i).
\end{align*}
Hence $X^*$ is in $\Cal B$.

\medskip
We now suppose that $X$ satisfies (2).
Fix $m$ and $\gamma$ as in (2).
We only treat the case when each $I_\beta$ is bounded, the other cases follow by slight modifications of our argument.
Let $f,f^*\colon R^{n-1}\to \overline{R}$ be the unique $\Sa R$-definable functions such that we have 
\[I_\beta = \{ \alpha\in R : f(\beta)<\alpha<f^*(\beta)\} \quad\text{for all  } \beta \in R^{n-1}.\]
Let $Z=\cl(\Gamma(f))$ and $Z^*=\cl(\Gamma(f^*)$.
By Lemma~\ref{lem:ind} $Z$ and $Z^*$ are both $\Sq R$-definable.
Note that $\Gamma(f)$ is the image of the function $Y \to Y \times \overline{R}$, $\alpha\mapsto(\alpha,f(\alpha))$, so by Lemma~\ref{lem:ind} we have $\dim Z\le\dim \cl(Y)$.
The same argument shows that $\dim Z^*\le \dim \cl(Y)$.
Applying the o-minimal fiber lemma we partition $\cl(Y)$ into disjoint $\Sq R$-definable sets $Y_0,Y_1$ such that $\dim Y_0 < \dim\cl(Y)$ and $|Z_\beta|,|Z^*_\beta|<\infty$ for all $\beta \in Y_1$.
By induction $X\cap\uppi^{-1}(Y_0)$ is in $\Cal B$, so it is enough to show that $X\cap\uppi^{-1}(Y_1)$ is in $\Cal B$.
Let $W_1,\ldots,W_k$ and $W^*_1,\ldots,W^*_k$ be partitions of $Y_1$ into pairwise disjoint sets and $g_i \colon W_i \to \R$, $g^*_i\colon W^*_i\to \R$, $i \in \{1,\ldots,k\}$ be $\Sq R$-definable functions such that $Z = \Gamma(g_1)\cup\cdots\cup\Gamma(g_k)$ and $Z^*=\Gamma(g^*_1)\cup\cdots\cup\Gamma(g^*_k)$.

\medskip
For each $i = 1,\ldots,k$ let $h_i \colon W_i \to \overline{R}$ and $h^*_i\colon W^*_i \to \overline{R}$ be given by
\begin{align*}
h_i(\alpha) &:= \{\beta \in R : \beta < g_i(\alpha) \}\\
h^*_i(\alpha) &:= \{\beta \in R : \beta < g^*_i(\alpha) \}\quad\text{for all  } \alpha\in R^{n-1}.
\end{align*}
Note that each $h_i,h^*_i$ is $\Sa R$-definable.
So for each $\beta \in Y_1 \cap \uppi(X)$ we have \[I_\beta=\{\alpha\in R : h_i(\beta)<\alpha<h^*_j(\beta)\}\quad\text{for some  }i,j\in\{1,\ldots,k\}.\]
For each $i,j\in \{1,\ldots,k\}$ let $Q_{ij}$ be the set of $\beta\in Y_1\cap\uppi(X)$ with this property.
Then $(Q_{ij} : i,j \in \{1,\ldots,k\})$ is a partition of $Y_1\cap \uppi(X)$ into $\Sa R$-definable sets.
It is enough to fix $i,j$ and show that $X':=X\cap\uppi^{-1}(Q_{ij})$ is in $\Cal B$.
By induction $Q_{ij}$ is in $\Cal B$.
Note that 
\begin{align*}
X'&= \{(\beta,\alpha)\in R^{n-1}\times R: \beta \in Q_{ij}, \alpha\in mR+\gamma, h_i(\beta)< \alpha < h^*_j(\beta)\}\\
&= [Q_{ij}\times R]\cap [R^{n-1}\times (mR+\gamma)]\cap \{(\beta,\alpha) \in \R^{n-1}\times \R : \beta \in W_i\cap W^*_j, g_i(\beta) < \alpha < g^*_j(\beta)\}.
\end{align*}
It follows that $X'$ is in $\Cal B$.
\end{proof}




We finally show that $\Sa R$ is trace equivalent to $\Sq R$.
As noted above $(\Sh R)^\square$ is interdefinable with $\Sq R$ and by Proposition~\ref{prop:she-0} $\Sa R$ is trace equivalent to $\Sh R$.
So after possibly replacing $\Sa R$ with $\Sh R$ we may suppose $\Sa R$ is Shelah complete.

\medskip
Now $\Th(\Sa R)$ trace defines $\Sq R$ as $\Sq R$ is interpretable in the Shelah completion of an elementary extension of $\Sa R$.
Proposition~\ref{prop:Z-decomp} shows that $(R; +)$ is trace definable in $\Th(\Sa R)$.
Hence it suffices to show that $\Sa R$ is trace definable in $(R; +) \sqcup \Sq R$.
Lemma~\ref{lem:combination} shows that the map $R \to R \times \R$ given by $a \mapsto (a,a)$ is a trace definition $\Sa R \rightsquigarrow  (R;+) \sqcup \Sq R$.

\subsection{The linear/field dichotomy}\label{section:dicho}
We continue to suppose that $\Sa R$ is a dp-minimal expansion of a substructure $(R;+,<)$ of $\rgoup$ with $R \ne \{0\}$.
We describe the dichotomy between linearity and field structure in this setting.
We first recall a special case of the Peterzil-Starchenko trichotomy~\cite{PS-Tri, Peterzil-Reducts}.
We let $\rvec$ be the ordered vector space of real numbers, i.e. $(\R;+,<,(x \mapsto \lambda x)_{\lambda \in \R})$.

\begin{fact}\label{fact:pz}
An o-minimal expansion of $\rgoup$ interprets $\rfield$ if and only if it interprets an infinite field if and only if it is not a reduct of $\rvec$.
\end{fact}

We now state the dichotomy.

\begin{corollary}\label{cor:final}
The following are equivalent:
\begin{enumerate}[leftmargin=*]
\item $\Sa R$ is locally trace equivalent to $\rgoup$.
\item $\Sa R$ is trace equivalent to an ordered vector space over a subfield of $\R$.
\item $\Th(\Sa R)$ does not locally trace define an infinite domain.
\item The Shelah completion of some $\aleph_1$-saturated elementary extension of $\Sa R$ does not interpret $\rfield$.
\item Every definable subset of every $R^n$ is a boolean combination of $(R;+)$-definable sets and sets of the form $Y \cap R^n$ for $\rvec$-definable $Y\subseteq \R^n$.
\item $\Sa R$ has near linear Zarankiewicz bounds.
\end{enumerate}
\end{corollary}

Can we replace infinite domains with infinite rings in (3)?
An application of Proposition~\ref{prop:field-blank} shows that $\Th(\Sa R)$ does not locally trace define a characteristic zero ring when $\Sa R$ has near linear Zarankiewicz bounds.
So it would be enough to handle positive characteristic rings.
Clearly, if $\Th(\Sa R)$ cannot locally trace define an infinite $\F_p$-vector space then $\Th(\Sa R)$ cannot locally trace define an infinite positive characteristic ring.
Conversely, the example discussed below Proposition~\ref{prop:field-blank} shows that if $\Th(\Sa R)$ locally trace defines an infinite $\F_p$-vector space then $\Th(\Sa R)$ locally trace defines an infinite positive characteristic ring.
I do not know if a dp-minimal expansion of a linear order can locally trace define an infinite $\F_p$-vector space.

\begin{proof}
Proposition~\ref{prop:lvoe} shows that (2) implies (1).
Fact~\ref{fact:mct} shows that (1) implies (6).
Proposition~\ref{prop:field-blank} and preservation of near linear Zarankiewicz bounds under local trace definability shows that (6) implies (3).
Furthermore (3) implies (4) by Proposition~\ref{prop:she-0}.

\medskip
We now show that (5) implies (6).
First note that (5) implies that $a \mapsto (a, a)$ is a trace definition $\Sa R \rightsquigarrow (R;+)\sqcup \rvec$.
Lemma~\ref{lem:vector space} shows that $(R;+)\sqcup \rvec$ has near linear Zarankiewicz bounds, hence $\Sa R$ has near linear Zarankiewicz bounds.

\medskip
We finish by showing that (4) implies both (5) and (2).
We first treat the case when $R$ is dense.
Then $\Sq R$ cannot interpret $\rfield$, hence $\Sq R$ is a reduct of $\rvec$ by Fact~\ref{fact:pz}.
Hence (5) holds by Lemma~\ref{lem:combination}.
By Theorem~\ref{thm:dp} $\Sa R$ is trace equivalent to $\Sq R$.
Furthermore $\Sq R$ is trace equivalent to $(\R;+,<,(x\mapsto \lambda x)_{\lambda \in F})$ for some subfield $F$ of $\R$ by Lemma~\ref{lem:ovs}.
Hence (4) implies (2).

\medskip
We now treat the case when $R$ is discrete.
After possibly rescaling we have $R = \Z$, so $\Sa R$ is interdefinable with $\zgoup$ by Fact~\ref{fact:simon}(4).
Hence (2) holds by Proposition~\ref{prop:big order}.
Finally, it follows from the usual description of definable sets in  $\zgoup$ that any definable subset of $\Z^n$ is a boolean combination of $(\Z;+)$-definable sets and sets of the form $Y \cap \Z^n$ for $\rgoup$-definable $Y \subseteq \R^n$.
\end{proof}

We prove one more result on dp-minimal expansions of archimedean ordered abelian groups.

\begin{corollary}\label{cor:final final}
Exactly one of the following holds:
\begin{enumerate}[leftmargin=*]
\item $\Sa R$ is locally trace equivalent to $\rgoup$.
\item $\Sa R$ is locally trace equivalent to an o-minimal expansion of $\rfield$.
\end{enumerate}
\end{corollary}

We have shown that $\Sa R$ is locally trace equivalent to $\Sq R$.
Hence it is enough to treat the case when $\Sa R$ is an o-minimal expansion of $\rgoup$.
We do this in the next section.




\subsection{O-minimal expansions of ordered groups}
We have seen that $\Th\rgoup$ cannot locally trace define an infinite field.
We prove the following.

\begin{proposition}\label{prop:o-min}
Any o-minimal expansion of an ordered abelian group is either locally trace equivalent to $\rgoup$ or locally trace equivalent to an o-minimal expansion of an ordered field.
Any o-minimal expansion of $\rgoup$ is either locally trace equivalent to $\rgoup$ or locally trace equivalent to an o-minimal expansion of $\rfield$.
\end{proposition}

We prove the real case and the general case of Proposition~\ref{prop:o-min} separately.
The real case also serves as a warm-up for the more difficult general case.
We first gather some more tools.
Recall that an o-minimal structure is \textbf{semi-bounded} if any definable function on a bounded interval is bounded.
Fact~\ref{fact:mund} is a theorem of Edmundo~\cite{ed-str}.

\begin{fact}
\label{fact:mund}
Suppose that $\Sa R$ is an o-minimal expansion of an ordered abelian group $(R;+,<)$.
Then the following are equivalent:
\begin{enumerate}
\item $\Sa R$ is not semi-bounded.
\item There are definable $\oplus,\otimes\colon R^2\to R$ such that $(R;\oplus,\otimes,<)\models\rcf$.
\end{enumerate}
\end{fact}

We say that $\Sa R$ is \textbf{linear} if it is a reduct of an ordered vector space with underlying ordered group $(R;+,<)$.
If $(R;+,<)=\rgoup$ then $\Sa R$ is linear if and only if $\Sa R$ is a reduct of $\rvec$.

\begin{fact}
\label{fact:lovey}
Suppose that $\Sa R$ is an o-minimal expansion of an ordered group $(R;+,<)$.
Then the following are equivalent:
\begin{enumerate}[leftmargin=*]
\item $\Sa R$ is not linear.
\item There is an interval $I \subseteq R$ and definable $\oplus,\otimes \colon I^2 \to I$ such that $(I;\oplus,\otimes,<) \models \rcf$.
\end{enumerate}
\end{fact}

Fact~\ref{fact:lovey} is a part of o-minimal  trichotomy~\cite{loveys,PS-Tri}.
Fact~\ref{fact:smib} is due to Edmundo~\cite{ed-str}.

\begin{fact}
\label{fact:smib}
Suppose that $\Sa R$ is a semi-bounded o-minimal expansion of an ordered abelian group $(R;+,<)$.
Let $\Cal B$ be the collection of bounded definable sets.
Then there is a unique ordered division ring $\D$ and ordered $\D$-vector space $\V$ expanding $(R;+,<)$ such that $\Sa R$ is interdefinable with $(\V,\Cal B)$ and $(\V,\Cal B)$ admits quantifier elimination.
\end{fact}

We now prove the real case of Proposition~\ref{prop:o-min}.
We show that an o-minimal expansion of $\rgoup$ is either locally trace equivalent to $\rgoup$ or locally trace equivalent to an o-minimal expansion of $\rfield$.

\begin{proof}
Let $\Sa R$ be an o-minimal expansion of $\rgoup$.
Suppose that $\Sa R$ is not semi-bounded.
Fix $\oplus,\otimes$ as in
Fact~\ref{fact:mund}.
Then $(\R;\oplus,\otimes,<)$ is a connected ordered field and is hence isomorphic to $(\R;+,\times,<)$.
Hence $\Sa R$ is bidefinable with an o-minimal expansion of $\rfield$.
We therefore suppose that $\Sa R$ is semi-bounded.
Proposition~\ref{prop:lvoe} shows that $\Sa R$ is locally trace equivalent to $\rgoup$ when $\Sa R$ is linear.
We therefore suppose that $\Sa R$ is non-linear.
Let $I,\oplus,\otimes$ be as in Fact~\ref{fact:lovey}.
Note that $I$ is bounded as $\Sa R$ is semi-bounded.
Again, $(I;\oplus,\otimes,<)$ is isomorphic to $(\R;+,\times,<)$.
Let $\Sa I$ be the structure induced on $I$ by $\Sa R$.
We show that $\Sa I$ is locally trace equivalent to $\Sa R$.
It is enough to show that $\Th(\Sa I)$ locally trace defines $\Sa R$.
Let $\Cal B$ and $\D$ be as in Fact~\ref{fact:smib}.
Then $(\V,\Cal B)=(\R;+,<,\Cal B,(x\mapsto\lambda x)_{\lambda\in\D})$ is interdefinable with $\Sa R$ and admits quantifier elimination.
Observe that $(x\mapsto\lambda x)_{\lambda\in\D}$ witnesses local trace definability of $(\V,\Cal B)$ in $(\R;+,<,\Cal B)$.
After rescaling and translating we suppose that $I \subseteq (0, 1)$.
Rescaling and translating we see that $(\R;+,<,\Cal B)$ is interdefinable with the expansion of $\rgoup$ by all $\Sa R$-definable subsets of all $[0, 1)^n$.
(This step fails in the non-archimedean setting.)
By Proposition~\ref{prop:nscover} $(\R;+,<,\Cal B)$ is trace equivalent to $\Sa I$.
\end{proof}

The remainder of this section is devoted to the proof that any o-minimal expansion of an ordered group is either locally trace equivalent to $\rgoup$ or locally trace equivalent to an o-minimal expansion of an infinite field.
We let $\Sa R$ be an o-minimal expansion of an ordered abelian group $(R;+,<)$.
Following the real case we   suppose that $\Sa R$ is semi-bounded and non-linear.
An interval $J\subseteq R$ is \textbf{short} if one of the following equivalent conditions holds:
\begin{enumerate}
\item For any positive $\delta\in R$ there is a definable surjection $[0,\delta]\to J$.
\item There are definable $\oplus,\otimes\colon J^2 \to J$ such that $(J;\oplus,\otimes,<)\models\rcf$.
\end{enumerate}
The equivalence follows by \cite[Cor.~3.3]{return}.
By semi-boundedness a short interval is bounded.
We say that $\Sa R$ is \textbf{shortnin'} if every bounded interval is short.

\medskip
Let $\Cal S$ be the collection of all $\Sa R$-definable sets that are contained in cartesian powers of short intervals.
Given a model $\Sa M$ of $\Th(\Sa R)$ we let $\Cal S_{\Sa M}$ be the collection of sets defined by the same formulas as $\Cal S$.
Fact~\ref{fact:short} is due to Peterzil~\cite{return}.

\begin{fact}
\label{fact:short}
Suppose $\Sa R$ is a non-linear semi-bounded o-minimal expansion of an ordered abelian group $(R;+,<)$ and  $\Cal S$ is as above.
Then there is an elementary extension $\Sa R\prec\Sa S$, an o-minimal semi-bounded expansion $\Sa N$ of $\Sa S$, and an elementary submodel $\Sa M$ of $\Sa N$ such that:
\begin{enumerate}
\item $\Sa M$ is shortnin'.
\item There is an ordered division ring $\D$ and an ordered $\D$-vector space $\V$ expanding $(S;+,<)$ such that $\Sa N=(\V,\Cal S_{\Sa S})$ and $\Sa N$ admits quantifier elimination.
\end{enumerate}
\end{fact}

Let $\Sa R, \Sa S, \Sa N, \Sa M$ be as in Fact~\ref{fact:short}.
Let $\Cal E$ be the collection of all bounded $\Sa M$-definable sets.
Note that if $J,J^*\subseteq M$ are bounded intervals and $f\colon J \to J^*$ is $\Sa M$-definable, then $f$ is $(M;+,<,\Cal E)$-definable as the graph of $f$ is in $\Cal E$.
Hence $(M;+,<,\Cal E)$ is shortnin'.
We show that $\Sa R$ is locally trace equivalent to $(M;+,<,\Cal E)$.
Note that $(M;+,<,\Cal S_{\Sa M})$ is reduct of $(M;+,<,\Cal E)$  and $(M;+,<,\Cal E)$ is a reduct of $\Sa M$.
Hence it is enough to show that $\Sa R$ is locally trace equivalent to both $\Sa M$ and $(M;+,<,\Cal S_{\Sa M})$.
We have $\Sa R\equiv\Sa S$, $\Sa N\equiv \Sa M$, and $(S;+,<,\Cal S_{\Sa S})\equiv (M;+,<,\Cal S_{\Sa M})$, so it suffices to show
 that $\Sa S$ is locally trace equivalent to both $\Sa N$ and $(S;+,<,\Cal S_{\Sa S})$.
 As $\Sa S$ is a reduct of $\Sa N$ and $(S;+,<,\Cal S_{\Sa S})$ is a reduct of $\Sa S$ it is enough to show that $(S;+,<,\Cal S_{\Sa S})$ is locally trace equivalent to $\Sa N$.
Now observe that $\Sa N$ is exactly $(S;+,<,\Cal S_{\Sa S},(x\mapsto\lambda x)_{\lambda \in\D})$ so  $(x \mapsto\lambda x)_{\lambda\in\D}$ witnesses local trace definability of $\Sa N$ in $(S;+,<,\Cal S_{\Sa S})$ by quantifier elimination for $\Sa N$.

\medskip
After possibly replacing $\Sa R$ with $(M;+,<,\Cal E)$ we suppose that $\Sa R$ is non-linear, semi-bounded, shortnin', and an expansion of $(R;+,<)$ by a collection $\Cal E$ of bounded sets.
Fix an interval $I\subseteq R$ and $\Sa R$-definable $\oplus,\otimes\colon I^2 \to I$ such that $(I;\oplus,\otimes,<)\models\rcf$.
Let $\Sa I$ be the structure induced on $I$ by $\Sa R$, considered as an expansion of $(I;\oplus,\otimes,<)$.
Then $\Sa I$ is an o-minimal expansion of an ordered field.
We show that $\Sa I$ is locally trace equivalent to $\Sa R$.
It is clear that $\Sa R$ interprets $\Sa I$, so it is enough to show that $\Th(\Sa I)$ locally trace defines $\Sa R$.
It is enough to show that any reduct of $\Sa R$ to a finite sublanguage is trace definable in $\Th(\Sa I)$.
Fix bounded $\Sa R$-definable $X_1, \ldots, X_m \subseteq R^n$.
We show that $(R; +, <, X_1, \ldots, X_n)$ is trace definable in $\Sa I$.
After possibly translating we suppose that each $X_i$ is contained in $[0, \gamma)^n$ for some positive $\gamma \in R$.
Proposition~\ref{prop:nscover} shows that $(R; +, <, X_1, \ldots, X_m)$ is trace definable in the theory of the structure induced on $[0, \gamma)$ by $\Sa R$.
Finally, note that the induced structure is interpretable in $\Sa I$ as $\Sa R$ is shortin'.

\section{Disintegrated o-minimal and weakly minimal structures}\label{section:dis}
We showed that an o-minimal expansion of an ordered abelian group trace defines an infinite field if and only if it interprets an infinite field.
We give an analogous result for groups.

\begin{proposition}\label{prop:dis}
An o-minimal structure trace defines an infinite group if and only if it interprets an infinite group.
\end{proposition}

Proposition~\ref{prop:dis} shows in particular that the Shelah completion of an o-minimal structure $\Sa M$ interprets an infinite group if and only if $\Sa M$ defines an infinite group.
This does not extend to weakly o-minimal structures by Proposition~\ref{prop:newdm}.

\medskip
By the Peterzil-Starchenko trichotomy theorem a non-disintegrated o-minimal structure defines an infinite group.
Hence it is enough to show that a disintegrated o-minimal structure cannot trace define an infinite group.
This was originally proven by the author in unpublished work, where it was also conjectured that monadically $\nip$ structures cannot trace define infinite groups.
The conjecture was proven by Laskowski and Braunfeld~\cite[Cor.~4.4]{monadic_nip}.

\begin{fact}\label{fact:lb}
A monadically $\nip$ theory cannot trace define an infinite group.
\end{fact}

Recall that a theory is {\bf monadically $\nip$} if any expansion of a model by unary relations is $\nip$ and a structure is monadically $\nip$ if its theory is.
By Fact~\ref{fact:monotone0} below any disintegrated o-minimal structure is monadically $\nip$.

\medskip
We showed above that any o-minimal expansion of an ordered group that does not interpret an infinite field is locally trace equivalent to $\rgoup$.
We now show that any o-minimal structure that does not interpret an infinite group is locally trace equivalent to $(\R;<)$.
We will see that this is a special case of the following more general theorem.

\begin{proposition}
\label{thm:2m}
Any weakly quasi-o-minimal structure in a binary relational language is locally trace equivalent to $(\Q;<)$.
\end{proposition}

A theory extending the theory of linear orders is \textbf{weakly quasi-o-minimal} if every definable unary set in every model is a boolean combination of zero-definable sets and convex sets.
We apply the following theorem of Moconja and Tanovi\'c~\cite[Thm.~2, Cor.~2.4]{Moconja}.

\begin{fact}
\label{fact:wqom}
Suppose that $\Sa M$ is a weakly quasi-o-minimal expansion of a linear order.
\begin{enumerate}
[leftmargin=*]
\item The reduct of $\Sa M$ generated by all definable subsets of $M^2$ admits quantifier elimination.
\item Every binary formula $\phi(x_1,x_2)$ is equivalent to a boolean combination of unary formulas $\theta(x_1), \theta(x_2)$ and binary formulas $\varphi(x_1,x_2)$ such that $\{\beta\in M :\Sa M\models \varphi(\alpha,\beta)\}$ is an initial segment of $M$ for all $\alpha \in M$.
\end{enumerate}
\end{fact}

We now prove Proposition~\ref{thm:2m}.

\begin{proof}
Let $\Sa M$ be a weakly quasi-o-minimal expansion of  a linear order $(M;\triangleleft)$ in a binary relational language $L$.
It is enough to show that $\Sa M$ is locally trace definable in $\dlo$.
Let  $(N;\triangleleft)$ be a dense complete linear order with a maximum and a minimum that extends $(M;\triangleleft)$.
We show that $(N;\triangleleft)$ locally trace defines $\Sa M$.
Let $L^*$ be the language containing a unary relation for every $\Sa M$-definable unary set and a binary relation $R_\varphi$ for every binary $L(M)$-formula $\varphi(x_1,x_2)$ such that $\{\beta\in M :\Sa M\models \varphi(\alpha,\beta)\}$ is an initial segment of $(M;\triangleleft)$ for all $\alpha \in M$.
Let $\Sa M^*$ be the natural $L^*$-structure on $M$.
By Fact~\ref{fact:wqom} $\Sa M^*$ is interdefinable with $\Sa M$. 
So it is enough to show that $(N;\triangleleft)$ locally trace defines $\Sa M^*$.
By Fact~\ref{fact:wqom} $\Sa M^*$ admits quantifier elimination.
Fix distinct $p,q\in N$ and for every unary $U\in L^*$ let $\chi_U\colon M\to \{p,q\}$ be given by declaring $\chi_U(a)=p$ if and only if $\Sa M^*\models U(a)$.
For every binary $R\in L^*$ let $\zeta_R\colon M\to N$ be given by $\zeta_R(\alpha)=\sup \{\beta\in M : \Sa M^*\models R(\alpha,\beta)\}$.
We show that the  $\chi_U$ and $\zeta_R$ witnesses local trace definability of $\Sa M^*$ in $(N;\triangleleft)$.
By Fact~\ref{fact:wqom}(2) and Proposition~\ref{prop:qe} it is enough to consider a fixed binary $R\in L^*$.
Let $U\in L^*$ be a unary relation defining the set of $\alpha \in M$ such that $\{\beta\in M : \Sa M^*\models R(\alpha,\beta)\}$ has a maximum in $M$.
For any $\alpha,\beta\in M$ we have $\Sa M^*\models R(\alpha,\beta)$ if and only if  either $\beta\triangleleft \zeta_R(\alpha)$ or $\beta=\zeta_R(\alpha)$ and $\chi_U(\alpha)=p$.
\end{proof}

We give some special cases of Proposition~\ref{thm:2m}.
Let $(M;\triangleleft)$ be a linear order. 
A binary relation $R\subseteq M \times J$  is increasing if $R(\alpha,\beta) \land (\alpha^* \triangleleft \alpha) \land (\beta \triangleleft \beta^*)$ implies $R(\alpha^*,\beta^*)$ and decreasing if $R(\alpha,\beta) \land (\alpha^* \triangleleft \alpha) \land (\beta^* \triangleleft \beta)$ implies $R(\alpha^*,\beta^*)$.
A binary relation is \textbf{monotone} if it is  increasing or decreasing\footnote{Previous authors, in particular Simon, use ``monotone" where we use ``increasing".}.
A {\bf monotone structure} is a structure which is interdefinable with a linear order equipped with unary relations and monotone binary relations.

\medskip
Fact~\ref{fact:monotone1} is also due to Moconja and Tanovi\'c~\cite[Cor.~2]{Moconja}.
The case of an expansion by unary relations and increasing relations is due to Simon~\cite[Prop.~4.1]{Simon-dp}.

\begin{fact}
\label{fact:monotone1}
Suppose that $\Sa M$ is a monotone $L$-structure.
Let $L^*$ be the language containing a unary relation for every zero-definable subset of $M$ and a binary relation for every zero-definable monotone binary relation.
Then $\Sa M$ admits quantifier elimination in $L^*$.
\end{fact}

It follows that a monotone structure is weakly quasi-o-minimal and binary.
Note also that any colored linear order is a monotone structure.
(In this case Fact~\ref{fact:monotone1} goes  back to Kamp's work on temporal logic~\cite{Kamp1968-KAMTLA}.)
Disintegrated o-minimal structures are monotone.

\begin{fact}\label{fact:monotone0}
The following are equivalent for any o-minimal structure $\Sa M$.
\begin{enumerate}[leftmargin=*]
\item $\Sa M$ is disintegrated.
\item $\Sa M$ is monadically $\nip$.
\item $\Sa M$ is monotone.
\item $\Sa M$ admits quantifier elimination in a binary relational language.
\item $\Sa M$ is interdefinable with the reduct generated by all definable subsets of $M^2$.
\end{enumerate}
\end{fact}

\begin{proof}
Laskowski and Braunfeld showed that (1) and (2) are equivalent \cite[Prop.~4.6]{monadic_nip}.
Mekler, Rubin, and Steinhorn showed that (1) and (4) are equivalent~\cite{mekler-rubin-steinhorn}.
Fact~\ref{fact:monotone1} shows that (3) implies (4) and (4) clearly implies (5).
Monotone structures are $\nip$ and are closed under expansions by unary relations, hence (3) implies (2).
If $\Sa M$ is o-minimal then an application of cell decomposition shows that any $\Sa M$-definable subset of $M^2$ is definable in the reduct of $\Sa M$ by all monotone functions $I \to M$ defined on intervals $I \subseteq M$.
Hence (5) implies (3).
\end{proof}


Corollary~\ref{cor:2m} follows by Proposition~\ref{thm:2m} and the remarks above.

\begin{corollary}
\label{cor:2m}
Any monotone structure is locally trace equivalent to $(\Q;<)$.
Infinite colored linear orders and  disintegrated o-minimal structures are~locally~trace~equivalent~to~$(\Q;<~)$.
\end{corollary}

\subsection{Disintegrated weakly minimal theories}\label{section:weak min}
We prove a stability-theoretic analogue of Proposition~\ref{thm:2m}.
Recall that $T$ is {\bf weakly minimal} if it is superstable of U-rank one.
Equivalently $T$ is weakly minimal if there is a small set $A$ of parameters from the monster model $\monster$ of $T$ such that every $\monster$-definable subset of $\monsterset$ is an $A$-definable set modulo a finite set~\cite[Thm.~21]{BPW-quasi}.
Let $R$ be a $k$-ary relation on a set $M$.
Then $R$ is \textbf{mutually algebraic} if there is $m$ such that for all $\alpha\in M$ there are at most $m$ elements $(\beta_1,\ldots,\beta_k)$ of $M^k$ such that $R(\beta_1,\ldots,\beta_k)$ and $\alpha\in\{\beta_1,\ldots,\beta_k\}$.
Let $\Sa M$ be an $L$-structure.
A formula $\phi(x_1,\ldots
,x_k)$ is mutually algebraic if it defines a mutually algebraic relation on $M$.
Furthermore $\Sa M$ is mutually algebraic if, up to interdefinability, $L$ is relational and every $R\in L$ is mutually algebraic.
If $R$ is a graph on $M$ then $R$ is mutually algebraic if and only if $R$ has bounded degree, so we view the class of mutually algebraic structures as a natural generalization of the class of bounded degree graphs.
Fact~\ref{fact:laskowski} is due to Laskowski~\cite{laskowski_2013}.

\begin{fact}\label{fact:laskowski}
A structure is mutually algebraic if and only if it is disintegrated and weakly minimal.
Furthermore if $\Sa M$ is mutually algebraic then every formula $\vartheta(x_1,\ldots,x_k)$ in $\Sa M$ is equivalent to a boolean combination of formulas of the form $\varphi(x_{i_1},\ldots,x_{i_n})$ for $n\le k$, $1\le i_1,\ldots,i_n\le k$, and a mutually algebraic formula $\varphi(y_1,\ldots,y_n)$ (possibly with parameters).
\end{fact}

With Fact~\ref{fact:laskowski} in place we can easily prove the following.

\begin{proposition}\label{prop:lask}
Any disintegrated weakly minimal theory is locally trace equivalent to the trivial theory of an infinite set with equality.
\end{proposition}

\begin{proof}
Suppose that $\Sa M$ is mutually algebraic.
We show that the trivial structure on $M$ locally trace defines $\Sa M$.
We may suppose that $\Sa M$ admits quantifier elimination in a relational language $L$ such that every $R\in L$ defines a mutually algebraic relation on $M$.
Then for every $k$-ary $R\in L$ there is $m_R$ such that $|\{\beta\in M^{k-1} : \Sa M\models R(\alpha,\beta)\}|\le m_R$ for all $\alpha \in M$.
Hence there is a function $\uptau^{i,j}_R\colon M\to M$ for each $k$-ary $R\in L$ and $i= 1,\ldots,m_R$ and $j = 1,\ldots,k-1$ such that we have the following for all $\alpha\in M$:
\[
\{\beta\in M^{k-1} : \Sa M\models R(\alpha,\beta)\} = \left\{(\uptau^{1,1}_R(\alpha),\ldots,\uptau^{1,k-1}_R(\alpha)),\ldots,(\uptau^{m_R,1}_R(\alpha),\ldots,\uptau^{m_R,k-1}_R(\alpha))\right\}.
\]
Let $\Cal E$ be the collection of all $\uptau^{i,j}_R$.
Then for any $k$-ary $R\in L$ and $\beta_1,\ldots,\beta_k\in M$ we have
\[
\Sa M\models R(\beta_1,\ldots,\beta_k)\quad\Longleftrightarrow\quad \bigvee_{i=1}^{m_R} \bigwedge_{j=1}^{k-1} \uptau^{i,j}_R(\beta_1) = \beta_{j+1}.
\]
Hence $\Cal E$ witnesses local trace definability of $\Sa M$ in the trivial structure on $M$.
\end{proof}

Corollary~\ref{cor:maa} is the ordered version of Proposition~\ref{thm:2m}.

\begin{corollary}
\label{cor:maa}
Any weakly minimal structure in a binary relational language is locally trace equivalent to an infinite set equipped with equality.
\end{corollary}

It is enough to show that any weakly minimal structure in a binary relational language is mutually algebraic.
This follows by the definitions and Proposition~\ref{prop:maa}.

\begin{proposition}
\label{prop:maa}
Suppose that $\Sa M$ is weakly minimal and let $\Sa M_\mathrm{ma}$ be the reduct of $\Sa M$ generated by all definable mutually algebraic relations.
Then every subset of $M^2$ that is definable in $\Sa M$ is already definable in $\Sa M_\mathrm{ma}$.
\end{proposition}

Given a set $X \subseteq A \times B$ we let $X^\beta = \{\alpha \in A : (\alpha, \beta) \in X\}$ for any $\beta \in B$.

\begin{proof}
After possibly passing to an elementary expansion and adding constants to the language we may suppose that every definable subset of $M$ is zero-definable modulo a finite set.
By  compactness there are  zero-definable sets $Y_1,\ldots,Y_m\subseteq M$ such that $X_\alpha$ agrees with some $Y_i$ modulo a finite set for every $\alpha \in M$.
We may suppose that $Y_i \triangle Y_j$ is infinite when $i\ne j$.
For each $i = 1,\ldots,m$ let:
\begin{enumerate}
\item  $Y^*_i$ be the set of $\alpha \in M$ such that $|X_\alpha\triangle Y_i|<\aleph_0$,
\item $W_i$ be the set of $(\alpha,\beta) \in M^2$ such that $\alpha \in Y^*_i$ and $\beta \in X_\alpha\setminus Y_i$, 
\item and $W'_i$ be the set of $(\alpha,\beta) \in M^2$ such that $\alpha \in Y^*_i$ and $\beta \in Y_i\setminus X_\alpha$.
\end{enumerate}
Note that $X = \bigcup_{i=1}^{m} \left([Y^*_i\times Y_i]\cup W_i\right)\setminus W'_i$.
Each $Y_i, Y^*_i$ is unary and hence definable in $\Sa M_\mathrm{ma}$.
Hence it suffices to show that each $W_i,W'_i$ is definable in $\Sa M_\mathrm{ma}$.

\medskip
We suppose that $W$ is a definable subset of $M^2$ such that $W_\alpha$ is finite for every $\alpha \in M$ and show that $W$ is definable in $\Sa M_\mathrm{ma}$.
Let $A$ be the set of $\alpha\in M$ such that $W^\alpha$ is infinite.
Then $A$ is finite as $\Sa M$ is a geometric structure.
Let $Y = \bigcup_{\alpha \in A} [W^\alpha\times\{\alpha\}]$.
Then $Y\subseteq W$ and $W\setminus Y$ is mutually algebraic.
As each $W^\alpha$ is unary $Y$ is definable in $\Sa M_\mathrm{ma}$, hence $W=Y\cup (W\setminus Y)$ is definable in $\Sa M_\mathrm{ma}$.
\end{proof}

\bibliographystyle{abbrv}
\bibliography{NIP}
\end{document}